\documentclass[11pt, reqno]{amsart}

\usepackage{amsmath, amssymb,amscd, amsthm,tikz-cd,mathrsfs}
\usepackage[vcentermath]{youngtab}
\usepackage{hyperref, calligra, extarrows}
\usepackage{enumitem}

\newcommand{\Addresses}{{% additional braces for segregating \footnotesize
  \bigskip
  \footnotesize

  \textsc{Department of Mathematics, University of Notre Dame,
    Notre Dame, IN 46556}\par\nopagebreak
  \textit{E-mail address}: \texttt{mperlman@nd.edu}

}}

\newcommand{\tn}{\textnormal}
\newcommand{\bw}{\bigwedge}
\newcommand{\Sym}{\operatorname{Sym}}

\newcommand{\defi}[1]{{\upshape\sffamily #1}}

\newtheorem {theorem} {Theorem}

\newtheorem {lemma}[theorem] {Lemma}
\newtheorem {corollary}[theorem]{Corollary}
\newtheorem {proposition}[theorem] {Proposition}

\newtheorem {definition}[theorem] {Definition}
\newtheorem*{theoremA} {Theorem F}
\newtheorem*{theoremB} {Theorem A}
\newtheorem*{theoremC} {Theorem B}

\newtheorem*{theoremE} {Theorem E}
\newtheorem*{theoremF} {Theorem C}
\newtheorem*{theoremG} {Theorem D}
\newtheorem*{theoremH} {Theorem G}

\theoremstyle{remark}
\newtheorem{remark}[theorem] {Remark}

\numberwithin{theorem}{section}
\numberwithin{equation}{section}

\title{Regularity and Cohomology of Pfaffian Thickenings}
\author{Michael Perlman}

\begin{document}

\begin{abstract}
Let $S$ be the coordinate ring of the space of $n\times n$ complex skew-symmetric matrices. This ring has an action of the group $\textnormal{GL}_n(\mathbb{C})$ induced by the action on the space of matrices. For every invariant ideal $I\subseteq S$, we provide an explicit description of the modules $\textnormal{Ext}^{\bullet}_S(S/I,S)$ in terms of irreducible representations. This allows us to give formulas for the regularity of basic invariant ideals and (symbolic) powers of ideals of Pfaffians, as well as to characterize when these ideals have a linear free resolution. In addition, given an inclusion of invariant ideals $I\supseteq J$, we compute the (co)kernel of the induced map $\textnormal{Ext}^j_S(S/I,S)\to \textnormal{Ext}^j_S(S/J,S)$ for all $j\geq 0$. As a consequence, we show that if an invariant ideal $I$ is unmixed, then the induced maps $\textnormal{Ext}_S^j(S/I,S)\to H_I^j(S)$ are injective, answering a question of Eisenbud-Musta\c{t}\u{a}-Stillman in the case of Pfaffian thickenings. Finally, using our Ext computations and local duality, we verify an instance of Kodaira vanishing in the sense described in the recent work of Bhatt-Blickle-Lyubeznik-Singh-Zhang.
\end{abstract}
\maketitle
\section{Introduction}

Let $W$ be a vector space of dimension $n$ over the complex numbers, and consider the polynomial ring $S=\Sym(\bw^2 W )$ with its natural structure as a representation of $\textnormal{GL}=\textnormal{GL}(W)$. The ideals $I\subseteq S$ invariant under this $\text{GL}$-action have been classified \cite{abeasis1980young}. Identifying $S$ with $\mathbb{C}[x_{i,j}]_{1\leq i<j\leq n}$, the ideal of $2k\times 2k$ Pfaffians of the generic skew-symmetric matrix $(x_{i,j})$, denoted by $I_{2k}\subseteq S$, defines the Pfaffian variety of matrices of rank $<2k$, while all other GL-invariant ideals of $S$ define non-reduced scheme structures (or \defi{thickenings}) of the Pfaffian variety, a class which includes powers and symbolic powers of ideals of Pfaffians. In this paper, motivated by the study of Castelnuovo-Mumford regularity and local cohomology, we determine the structure of $\textnormal{Ext}^{j}_S(S/I,S)$ as a representation of $\tn{GL}$ for all invariant ideals $I\subseteq S$ and all $j\geq 0$ (see Theorem F). As a consequence of our computations, one may obtain the regularity and projective dimension of any ideal defining an equivariant Pfaffian thickening, two invariants which bound the shape of the graded Betti table of the ideal. One instance where we obtain closed formulas for the regularity is in the case of large powers and symbolic powers of a Pfaffian ideal. Our theorem on regularity of powers is stated below and can be found in Section 6. We write $I^{\textnormal{sat}}$ for the saturation of $I$ with respect to the irrelevant ideal, and $I^{(d)}$ for the $d$-th symbolic power of $I$ (for definitions, see Section \ref{prelim}).

\begin{theoremB}
Let $2<2k\leq n-2$. If one of the following holds: (1) $n$ is even and $d\geq n-2$, (2) $n$ is odd and $d\geq n-3$, then
\begin{equation}\label{lineareqReg}
\textnormal{reg}\left(I_{2k}^{(d)}\right)=dk,\textnormal{ and }\textnormal{reg}\left(I_{2k}^d\right)=\textnormal{reg}\left(\left(I_{2k}^d\right)^{\textnormal{sat}}\right)=dk+
\begin{cases}
k\left(\frac{k}{2}-1\right) & \textnormal{ if $k$ is even,}\\
\frac{1}{2}(k-1)^2 & \textnormal{ if $k$ is odd.}
\end{cases}
\end{equation}
Further, if neither of these conditions hold, then
\begin{equation}\label{badregMain}
\textnormal{reg}\left(I_{2k}^d\right)\geq \textnormal{reg}\left(\left(I_{2k}^d\right)^{\textnormal{sat}}\right)>dk.
\end{equation}
Finally, if $d\leq n-4$ then $\textnormal{reg}(I_{2k}^{(d)})>dk$.
\end{theoremB}

Our desire to provide a closed formula for the regularity of powers of Pfaffian ideals is spawned by the seminal results \cite[Theorem 1.1]{cutkosky1999asymptotic} and \cite[Theorem 5]{kodiyalam2000asymptotic}. Cutkosky-Herzog-Trung and Kodiyalam independently showed that if $I$ is any homogeneous ideal in a standard graded polynomial ring over an arbitrary field, then the regularity of the $d$-th power $I^d$ is a linear function of $d$ for $d$ sufficiently large. Theorem A demonstrates the linear function whose existence is implied by their results. 

When $I=I_{2k}$ is a Pfaffian ideal, the syzygies (and thus the regularity) are known by the work of J\'{o}zefiak-Pragacz-Weyman \cite{jozefiak1981resolutions}. When $I=I_{n-1}^d$ is any power of the ideal of sub-maximal Pfaffians, the syzygies were known by Boffi-S\'{a}nchez \cite{boffi1992resolutions} and independently by Kustin-Ulrich \cite{kustin1992family}. In all cases, we recover the previous results on regularity using our Ext computations. Theorem A allows us to determine when a power of a Pfaffian ideal has linear minimal free resolution:

\begin{theoremC}
Consider an integer $d\geq 1$ and let $2\leq 2k\leq n$. Then $I_{2k}^d$ has linear minimal free resolution if and only if one of the following holds:
\begin{enumerate}
\item $2k=2$ or $2k=n$, and $d\geq 1$,
\item $2k=n-1$, and $d$ is even or $d\geq n-3$,
\item $n$ is even, $2k=4$, and $d\geq n-2$,
\item $n$ is odd, $2k=4$, and $d\geq n-3$.

\end{enumerate}
\end{theoremC}

\noindent When $2k=2$, $I_{2k}^d$ is a power of the homogeneous maximal ideal, and when $2k=n$, $I_{2k}^d$ is the principal ideal generated by the $d$-th power of the $n\times n$ Pfaffian. Thus, case (1) is classical. Case (2) follows from the previous work \cite{boffi1992resolutions, kustin1992family}  described above. Our contribution is showing that $I_{2k}^d$ has linear minimal free resolution in cases (3) and (4), as well as showing that this list is exhaustive.

Every GL-invariant ideal in $S$ is a sum of \defi{basic invariant ideals}, written $I_{\underline{x}}$, where $\underline{x}$ is a partition with at most $\lfloor n/2 \rfloor$ parts (see Section \ref{prelim} for definition and properties). Familiar examples of these ideals include Pfaffian ideals, powers of the ideal generated by the $n\times n$ Pfaffian, and powers of ideals of sub-maximal Pfaffians. Besides the previously mentioned computations \cite{jozefiak1981resolutions, boffi1992resolutions, kustin1992family}, it remains an open problem to compute the syzygies of the basic invariant ideals. As another application of our Ext computations, we obtain formulas for the regularity of the basic invariant ideals. In the statement of the theorems we write $\underline{x}'$ for the conjugate partition of the partition $\underline{x}$.

\begin{theoremF}
Let $\underline{x}$ be a partition with at most $\lfloor n/2 \rfloor$ parts. If $n$ is even then
$$
\textnormal{reg}(I_{\underline{x}})=\max_{0\leq c\leq x_1-1} \left( (n-2x_{c+1}'+1)(x_{c+1}'-1)+cn/2+1\right).
$$
If $n$ is odd then
$$
\textnormal{reg}(I_{\underline{x}})=\max_{0\leq c\leq x_1-1} \left( b(x_{c+1}'-1,n,c)+1\right),
$$
where the function $b:\mathbb{Z}^3\to \mathbb{Z}$ is defined by
\begin{equation}\label{bfunction}
b(l,n,c)=
\begin{cases}
(\lfloor n/2 \rfloor-l)(2l+c)+l(c-1) &  \textnormal{ if $c>2l$, or $c< 2l$ and $c$ is odd}\\
(\lfloor n/2 \rfloor-l)(2l+c)+c(l-1/2) & \textnormal{ if $c\leq 2l$ and $c$ is even.}
\end{cases}
\end{equation}
\end{theoremF}

These computations are the content of Section 4. The above formula allows us to characterize when a basic invariant ideal has a linear minimal free resolution:

\begin{theoremG}
If $n$ is even, then $I_{\underline{x}}$ has linear minimal free resolution if and only if $\underline{x}'=((n/2)^k,1^l)$, where $k\geq 0$ and $l\in \{0,1\}$. If $n$ is odd, then $I_{\underline{x}}$ has linear minimal free resolution if and only if one of the following holds:
\begin{enumerate}
\item $\underline{x}'=(\lfloor n/2 \rfloor^k)$ for $k\geq n-2$, or $k$ is even and $k\leq n-3$,
\item $\underline{x}'=(\lfloor n/2 \rfloor^k,1)$ for $k\geq n-5$, or $k$ is even and $k\leq n-7$.
\end{enumerate}
\end{theoremG}

In another direction, via graded local duality, we use our Ext computations to study cohomological vanishing properties of twists of the line bundle embedding an equivariant Pfaffian thickening $Y\subset \mathbb{P}(\bw^2 W^{\ast})$. In \cite{bhatt2016stabilization}, the authors prove a version of the Kodaira vanishing theorem for thickenings defined by a power of the ideal sheaf of a local complete intersection. Further, they show that their results do not hold for a general thickening. Nonetheless, we verify that an analogue of their vanishing theorem holds for all equivariant Pfaffian thickenings, similar to the result for generic determinantal thickenings \cite[Theorem 6.1]{raicu2016regularity}:

\begin{theoremE}
Let $I\subseteq S$ be a GL-invariant ideal and $Y\subset \mathbb{P}(\bw^2 W^{\ast})$ the thickening that it defines. Then
\begin{equation}\label{firstassertion}
H^q(Y,\mathcal{O}_Y(-j))=0\textnormal{ for $q<2n-4$ and $j>0$}.
\end{equation}
In particular, if we let $Y_{\textnormal{red}}$ denote the underlying Pfaffian variety, and if we make the convention that $\textnormal{codim}(\textnormal{Sing}(Y_{\textnormal{red}}))=\textnormal{dim}(Y_{\textnormal{red}})$ when $Y_{\textnormal{red}}$ is non-singular, then
$$
H^q(Y,\mathcal{O}_Y(-j))=0\textnormal{ for $q<\textnormal{codim}(\textnormal{Sing}(Y_{\textnormal{red}}))$ and $j>0$.}
$$
\end{theoremE}

Write $Y_{2k}\subseteq \mathbb{P}(\bw^2W^{\ast})$ for the vanishing locus of the ideal $I_{2k}\subseteq S$. To make the statement of Theorem E more clear, note that all equivariant Pfaffian thickenings $Y\subseteq \mathbb{P}(\bw^2 W^{\ast})$ are supported on Pfaffian varieties, i.e. $Y_{\tn{red}}=Y_{2k}$ for some $2\leq k \leq n/2$. The variety $Y_{4}$ is the Grassmannian $\mathbb{G}(2,n)$ under the Pl\"{u}cker embedding, so it is smooth, and $\tn{codim}(\tn{Sing}(Y_4))=\dim(Y_4)=2n-4$ by convention. In general, for $k\geq 3$, the singular locus of $Y_{2k}$ is $Y_{2k-2}$, so $\tn{codim}(\tn{Sing}(Y_{2k}))=\tn{codim}(Y_{2k-2},Y_{2k})$ in this case.

The above computations are all consequences of (\ref{extformula}) in the following theorem, as well as Theorem \ref{extjxp}, recalling that for any finitely-generated graded $S$-module $M$, we can determine regularity with knowledge of the graded pieces of Ext modules: 
\begin{equation}\label{compreg}
\textnormal{reg}(M)=\max\left\{ -r-j \mid \textnormal{Ext}^j_S(M,S)_r\neq 0\right\},
\end{equation}
and $\textnormal{reg}(I)=\textnormal{reg}(S/I)+1$ for any homogeneous ideal $I\subseteq S$.

\begin{theoremA}\label{maintheoremext}
To any $\text{GL}$-invariant ideal $I\subseteq S$ we can associate a finite set $\mathcal{M}(I)$ of $\text{GL}$-equivariant $S$-modules with the property that for each $j\geq 0$
\begin{equation}\label{extformula}
\textnormal{Ext}^j_S(S/I,S)\cong \bigoplus_{M\in \mathcal{M}(I)} \textnormal{Ext}^j_S(M,S),
\end{equation}
where the above isomorphism is $\text{GL}$-equivariant and degree preserving, but is not necessarily an isomorphism of $S$-modules. The sets $\mathcal{M}(I)$ and the modules $\textnormal{Ext}^j_S(M,S)$ for $M\in \mathcal{M}(I)$ can be computed explicitly. Furthermore, the association $I\mapsto \mathcal{M}(I)$ has the property that whenever $I\supseteq J$ are $\tn{GL}$-invariant ideals, the (co)kernels and images of the induced maps $\textnormal{Ext}^j_S(S/I,S)\to \textnormal{Ext}^j_S(S/J,S)$ are computed as follows.
\begin{equation}\label{kernel}
\textnormal{ker}\left(   \textnormal{Ext}_S^j(S/I,S)\to \textnormal{Ext}_S^j(S/J,S)  \right)\cong \bigoplus_{M\in \mathcal{M}(I)\setminus \mathcal{M}(J)} \textnormal{Ext}^j_S(M,S),
\end{equation}
\begin{equation}\label{image}
\textnormal{Im}\left(   \textnormal{Ext}_S^j(S/I,S)\to \textnormal{Ext}_S^j(S/J,S)  \right)\cong \bigoplus_{M\in \mathcal{M}(I)\cap \mathcal{M}(J)} \textnormal{Ext}^j_S(M,S),
\end{equation}
\begin{equation}\label{cokernel}
\textnormal{coker}\left(   \textnormal{Ext}_S^j(S/I,S)\to \textnormal{Ext}_S^j(S/J,S)  \right)\cong \bigoplus_{M\in \mathcal{M}(J)\setminus \mathcal{M}(I)} \textnormal{Ext}^j_S(M,S).
\end{equation}
\end{theoremA}
\noindent In order to make the statement of Theorem F more precise, we make a couple remarks. There is a $\textnormal{GL}$-equivariant filtration of $S/I$ by $S$-modules, and the $M$ appearing in $\mathcal{M}(I)$ are the factors of this filtration. These $M$ are among the quotients $J_{\underline{z},l}$ of GL-invariant ideals defined in (\ref{J}) below, and were first studied by Raicu-Weyman in order to compute local cohomology with support in Pfaffian varieties \cite{raicu2016local}. The sets of pairs $(\underline{z},l)$ for which $M=J_{\underline{z},l}$ belongs to $\mathcal{M}(I)$ for a given $\tn{GL}$-invariant ideal are described in Definition \ref{subquotset}, and the computations of $\tn{Ext}^j_S(J_{\underline{z},l},S)$ are done in Theorem \ref{extjxp} (for more details, see Section \ref{mainsub}).

This project was inspired by the recent work of Raicu, who completed the similar Ext calculations in the case of generic $n \times m$ matrices with a $\textnormal{GL}_n(\mathbb{C})\times \textnormal{GL}_m(\mathbb{C})$ action, and applied these calculations to study the regularity of powers and symbolic powers of generic determinantal ideals \cite{raicu2016regularity}. Our method for computing Ext is analogous to the method of Raicu, and is an application of the ``geometric technique for Ext" \cite[Theorem 3.1]{raicu2014submax}, which combines Grothendieck duality with the use of desingularizations via vector bundles on projective varieties, as in the Kempf-Lascoux-Weyman geometric technique for syzygies \cite[Chapter 5]{weyman2003cohomology}.

We now discuss applications of (\ref{kernel}), (\ref{image}), and (\ref{cokernel}) to local cohomology with support in Pfaffian varieties. Using explicit descriptions of $\mathcal{M}(I)$ for $I=I_{2k}^d$ and $I=I_{2k}^{(d)}$ (see (\ref{Zpfaffequation}),(\ref{symbolicdirectequation})), part (\ref{kernel}) above yields:

\begin{theoremH}\label{growthofext}
For $2\leq 2k\leq n$ and $d\geq 1$ consider the inclusions $I_{2k}^d\subseteq I_{2k}^{(d)}$ and $I_{2k}^{(d+1)}\subseteq I_{2k}^{(d)}$. For each $j\geq 0$, the induced maps of Ext modules
$$
\textnormal{Ext}^j_S(S/I_{2k}^{(d)},S)\longrightarrow \textnormal{Ext}^j_S(S/I_{2k}^{d},S)\;\textnormal{   and   }\;\textnormal{Ext}^j_S(S/I_{2k}^{(d)},S)\longrightarrow \textnormal{Ext}^j_S(S/I_{2k}^{(d+1)},S),
$$
are injective.
\end{theoremH}

\noindent Since the sequence of ideals $\{I_{2k}^{(d)}\}_{d\geq 1}$ is cofinal to the sequence $\{I_{2k}^d\}_{d\geq 1}$, it follows from \cite[Remark 7.9]{iyengar2007twenty} that the local cohomology of $S$ with support in any Pfaffian variety may be described as the union:
\begin{equation}
H^j_{I_{2k}}(S)=\bigcup_{d\geq 1} \tn{Ext}^j_S(S/I_{2k}^{(d)},S).
\end{equation}

\noindent With a little more work, we show that if $I\subset S$ is an invariant ideal, the natural maps $\tn{Ext}_S^j(S/I,S)\to H^j_I(S)$ are injective if $I$ is unmixed, answering a question of Eisenbud--Musta\c{t}\u{a}-Stillman in the case of Pfaffian thickenings \cite[Question 6.2]{eisenbud2000cohomology}. This result serves as a converse to their result that unmixedness is necessary \cite[Example 6]{eisenbud2000cohomology}. In more recent work \cite{perlman2018lyubeznik}, we use (\ref{extformula}) and Theorem \ref{extjxp} to determine the filtration of the local cohomology $H^j_{I_{2k}}(S)$ as a module over the Weyl algebra $\mathcal{D}_X$, expanding upon the previously mentioned computations of Raicu-Weyman \cite[Main Theorem]{raicu2016local}. With knowledge of the $\mathcal{D}_X$-module filtration of these local cohomology modules, we obtain the Lyubeznik numbers (see \cite[Section 4]{MR1223223}) for Pfaffian rings via careful use of graded local duality, using all parts of Theorem F above.\\

\noindent \textbf{Organization.} In Section \ref{prelim} we review the basics of representation theory of $\tn{GL}(W)$, as well as set notation for the objects of study. In Section \ref{Section3} we compute the Ext modules of the subquotients $J_{\underline{z},l}$, as well as describe the differences between the proof of Theorem F and the proof of \cite[Theorem 3.3]{raicu2016regularity}. In Section \ref{Section4} we study regularity of the basic thickenings, and in Section \ref{optimization} we solve an optimization problem in order to compute the regularity of (symbolic) powers of Pfaffian ideals in Section \ref{proofReg}. We end Section \ref{Section5} by proving Theorem G and discussing its applications. In Section \ref{Section6} we prove the Kodaira vanishing result Theorem E. Finally, in Section \ref{Section7} we carry out some examples of Ext and regularity computations for thickenings of the Pl\"{u}cker-embedded Grassmannian $\mathbb{G}(2,6)\subseteq \mathbb{P}(\bw^2 \mathbb{C}^6)$.

\section{Preliminaries}\label{prelim}

Let $W$ be an $n$-dimensional complex vector space. The irreducible representations of $\text{GL}=\textnormal{GL}(W)$ are classified by $\mathbb{Z}^n_{\text{dom}}$, the set of dominant weights $\lambda=(\lambda_1 \geq \cdots \geq \lambda_n)\in \mathbb{Z}^n$. We write $S_{\lambda}W$ for the irreducible representation corresponding to $\lambda$, where $S_{\lambda}(-)$ is the \defi{Schur functor} associated to $\lambda$. In Lemma \ref{simpext} we will consider Schur functors applied to the tautological sub and quotient bundles on the Grassmannian. For ease of notation throughout, we write $\text{det}(W)=\bigwedge^n W$ for the irreducible representation $S_{\lambda}W$ corresponding to $\lambda=(1,\ldots ,1)=(1^n)$. The size of a weight $\lambda$ is $|\lambda|=\lambda_1+\cdots +\lambda_n$. A partition is a dominant weight whose entries are non-negative integers, and we will use underlined roman letters for partitions to distinguish them from the arbitrary dominant weights. We will identify a partition $\underline{x}$ with the associated Young diagram:
$$
\underline{x}=(4,3,1,0)\;\; \longleftrightarrow \;\;\;\yng(4,3,1)
$$
\vspace{.01cm}

\noindent When we refer to a row or column of $\underline{x}$, we mean a row or column of the associated Young diagram. Given a partition $\underline{x}$, we construct the conjugate partition $\underline{x}'$ by transposing the associated Young diagram. In other words, $x_i'$ is the number of boxes in the $i$-the column of $\underline{x}$. Given a positive integer $c$, we write $\underline{x}(c)$ for the partition defined by $\underline{x}(c)_i=\min (x_i,c)$. The Young diagram of $\underline{x}(c)$ is the first $c$ columns of the Young diagram of $\underline{x}$. We let $\mathcal{P}(k)$ denote the set of partitions $\underline{z}=(z_1\geq \cdots \geq z_k\geq 0)$, and write $\mathcal{P}_e(k)$ for the partitions with even column lengths in their Young diagrams, i.e if $\underline{x}\in \mathcal{P}_e(k)$, then $x_{2i}=x_{2i-1}$ for all $i=1,\cdots, \lfloor k/2 \rfloor$, and $x_k=0$ if $k$ is odd. Given $\underline{z}\in \mathcal{P}(k)$ we write
$$
\underline{z}^{(2)}=(z_1,z_1,z_2,z_2, \cdots)\in \mathcal{P}_e(2k).
$$
Whenever convenient, we will identify these sets as subsets of each other, i.e. $\mathcal{P}(k)\subset \mathcal{P}(k+1)$. We put a partial ordering $\leq$ on $\mathcal{P}(k)$ as follows: given $\underline{x},\underline{y}\in \mathcal{P}(k)$, write $\underline{x}\leq \underline{y}$ if $x_i\leq y_i$ for all $i=1,\cdots, k$. In other words, $\underline{x}\leq \underline{y}$ if and only if the Young diagram of $\underline{x}$ fits inside the Young diagram of $\underline{y}$.

For ease of notation, we will write $m=\lfloor n/2 \rfloor$ throughout. Let $S=\Sym(\bw^2 W )$ be the ring of polynomial functions on the space on $n\times n$ skew-symmetric matrices.  By \cite[Proposition 2.3.8]{weyman2003cohomology}, we have the formula:
\begin{equation}\label{cauchy}
S=\bigoplus_{\underline{z}\in \mathcal{P}(m)} S_{\underline{z}^{(2)}}W.
\end{equation}
Identifying $S$ with $\mathbb{C}[x_{i,j}]_{1\leq i<j\leq n}$, we define for each even integer $0\leq q \leq n$ the polynomial $\textnormal{Pfaff}_q=\textnormal{Pfaff}(x_{i,j})_{1\leq i< j \leq q}$, the top left $q\times q$ Pfaffian of the $n\times n$ skew-symmetric matrix of indeterminates. For $\underline{z}\in \mathcal{P}(m)$ we define
\begin{equation}\label{highestweight}
\textnormal{Pfaff}_{\underline{z}}=\prod_{i=1}^{z_1} \textnormal{Pfaff}_{2z_i'}.
\end{equation}
Using this notation, the subrepresentation $S_{\underline{z}^{(2)}}W\subset S$ has highest weight vector $\tn{Pfaff}_{\underline{z}}$. In particular, the irreducible representation $S_{\underline{z}^{(2)}}W\subset S$ is the linear span of the GL-orbit of the polynomial $\textnormal{Pfaff}_{\underline{z}}$.

We are now ready to recall the classification of $\tn{GL}$-invariant ideals in $S$ due to Abeasis-Del Fra \cite{abeasis1980young}. For any partition $\underline{z}\in \mathcal{P}(m)$ we define the \defi{basic invariant ideal} associated to $\underline{z}$:
\begin{equation}\label{ilambda}
I_{\underline{z}}:=\textnormal{ the ideal in $S$ generated by $S_{\underline{z}^{(2)}}W$}.
\end{equation}
Recalling the partial ordering $\leq$ on partitions introduced above, the basic invariant ideal $I_{\underline{z}}$ has the following decomposition as a representation of $\tn{GL}$:
\begin{equation}\label{idealDecomp}
I_{\underline{z}}=\bigoplus_{\underline{y}\geq \underline{z}} S_{\underline{y}^{(2)}} W.
\end{equation}
It follows immediately that $I_{\underline{y}}\subseteq I_{\underline{z}}$ if and only if $\underline{z}\leq \underline{y}$. We use the basic invariant ideals to construct all $\tn{GL}$-invariant ideals of $S$ as follows: given a set of partitions $\mathcal{X}\subset \mathcal{P}(m)$, there is an associated invariant ideal
\begin{equation}\label{sumofbasics}
I_{\mathcal{X}}:=\sum_{\underline{x}\in \mathcal{X}} I_{\underline{x}}.
\end{equation}
Since any $\tn{GL}$-invariant ideal in $S$ is a direct sum of irreducible representations appearing in the decomposition (\ref{cauchy}), it follows that all invariant ideals are finite sums of basic invariant ideals. In other words, every $\tn{GL}$-invariant ideal in $S$ is of the form $I_{\mathcal{X}}$ for some finite subset $\mathcal{X}\subset \mathcal{P}(m)$. We note that the correspondence between invariant ideals and finite sets of partitions is not one-to-one: if $\underline{x},\underline{y}\in \mathcal{X}$ and $\underline{x}\leq \underline{y}$, then $I_{\mathcal{X}}=I_{\mathcal{X}-\{\underline{y}\}}$. Indeed, this follows from (\ref{sumofbasics}) and the decomposition (\ref{idealDecomp}).

When $I\subseteq S$ is a power or symbolic power of a Pfaffian ideal, there is an explicit description of the set of partitions $\mathcal{X}\subset \mathcal{P}(m)$ for which $I=I_{\mathcal{X}}$. Given integers $k$ and $d$ with $1\leq k \leq m$ and $d\geq 1$, we write
\begin{equation}\label{subscriptpfaff}
\mathcal{X}_{k}^d=\{\underline{x}\in \mathcal{P}(m)\mid |\underline{x}|=kd,\; x_1\leq d\}.
\end{equation}
By \cite[Theorem 4.1]{abeasis1980young}, the $d$-th power of the Pfaffian ideal $I_{2k}$ is equal to $I_{\mathcal{X}^d_{k}}$. In particular, setting $\underline{z}=(1^k)$, we have that $I_{2k}=I_{\underline{z}}$, i.e. Pfaffian ideals are basic. We now introduce the sets of partitions that index symbolic powers of Pfaffian ideals. Recall that the $d$-th symbolic power of a prime ideal $I$, written $I^{(d)}$, is the ideal of polynomial functions that vanish to order $\geq d$ at every point in the affine vanishing locus of $I$. Consider the set of partitions
\begin{equation}\label{subscriptsymbolic}
\mathcal{X}_{k}^{(d)}=\{\underline{x}\in \mathcal{P}(m) \mid x_1=\cdots =x_{k},\; x_{k}+\cdots + x_{m}=d\}.
\end{equation}
By \cite[Theorem 5.1]{abeasis1980young} the $d$-th symbolic power $I_{2k}^{(d)}$ is equal to $I_{\mathcal{X}_{k}^{(d)}}$. Next, we recall the \defi{saturation} of an ideal $I$ with respect to the homogeneous maximal ideal $\mathfrak{m}\subset S$:
\begin{equation}\label{saturation}
I^{\tn{sat}}=\{f\in S \mid f\cdot \mathfrak{m}^d \subseteq I\tn{  for  }d\gg 0\}.
\end{equation}
By a proof identical to the proof of \cite[Lemma 2.3]{raicu2016regularity} we describe the saturation of any power of a Pfaffian ideal (using notation consistent with \cite[Equation 2.14]{raicu2016regularity}):
\begin{equation}\label{satPfaff}
(I_{2k}^d)^{\tn{sat}}=I_{\mathcal{X}_k^{d:1}},\;\tn{where}\;\;\;\; \mathcal{X}_k^{d:1}=\{ \underline{x}(c)\mid \underline{x}\in \mathcal{X}^d_k, c\in \mathbb{Z}_{\geq 0}, x'_c>1 \tn{ if } c>0,\tn{ and }x_{c+1}'\leq 1\}.
\end{equation}
In Section \ref{proofReg} we use the descriptions (\ref{subscriptpfaff}), (\ref{subscriptsymbolic}), (\ref{satPfaff}) to study the regularity of $I_{2k}^d$, $I_{2k}^{(d)}$, and $(I_{2k}^d)^{\tn{sat}}$.

Finally, we define the $\tn{GL}$-equivariant $S$-modules $J_{\underline{z},l}$ introduced by Raicu-Weyman \cite[Lemma 2.5]{raicu2016local}. These modules appear as subquotients of the homogeneous coordinate ring $S/I$ of an equivariant Pfaffian thickening, and they constitute the sets $\mathcal{M}(I)$ in the statement of Theorem F. Given $0\leq l\leq m-1$ and $\underline{z}\in \mathcal{P}(m)$ with $z_1=\cdots =z_{l+1}$, we consider the collection of partitions obtained from $\underline{z}$ by adding boxes to its Young diagram in row $l+1$ or higher (and possibly adding boxes elsewhere):
\begin{equation}\label{succ}
\mathfrak{succ}(\underline{z},l)=\left\{ \underline{y}\in \mathcal{P}(m) \mid \underline{y}\geq \underline{z}\;\; \textnormal{and} \;\; y_i>z_i\;\; \textnormal{for some}\;\; i>l\right\}.
\end{equation}
By (\ref{idealDecomp}) we have that $I_{\mathfrak{succ}(\underline{z},l)}\subseteq I_{\underline{z}}$. The module $J_{\underline{z},l}$ is defined to be the quotient of invariant ideals:
\begin{equation}\label{J}
J_{\underline{z},l}:=I_{\underline{z}}/I_{\mathfrak{succ}(\underline{z},l)}.
\end{equation}
Our notation is different from \cite{raicu2016local}, where the authors write $J_{\underline{z}^{(2)},l}$ for the module we call $J_{\underline{z},l}$.

\section{Ext Modules for the $\textnormal{GL}$-invariant Ideals }\label{Section3}

Let $S$ be the ring of polynomial functions on the space of $n\times n$ skew-symmetric matrices, and continue to write $m=\lfloor n/2 \rfloor$. In this section we describe the proof of Theorem F. We begin by computing $\textnormal{Ext}^{\bullet}_S(J_{\underline{z},l},S)$ for all $0\leq l\leq m-1$ and $\underline{z}\in \mathcal{P}(m)$ with $z_1=\cdots =z_{l+1}$. 

\subsection{Ext Modules of the Subquotients $J_{\underline{z},l}$ and Regularity}\label{Section3.1}

We first prove a statement which allows us to compute the $\tn{GL}$-structure of the modules $\tn{Ext}^j_S(J_{\underline{z},l},S)$ using Bott's Theorem for Grassmannians \cite[Corollary 4.1.9]{weyman2003cohomology}. The main result of this subsection is Theorem \ref{extjxp} below. These computations were set up in \cite[Lemma 2.5]{raicu2016local} using the geometric technique for Ext \cite[Theorem 3.1]{raicu2014submax}, but the authors did not require the entire $\tn{GL}$-equivariant structure of $\textnormal{Ext}^{\bullet}_S(J_{\underline{z},l},S)$ for their purposes of computing local cohomology. We conclude the subsection by deriving a formula for the regularity of the modules $J_{\underline{z},l}$. Given an integer $1\leq r\leq n-1$, we write $\mathbb{G}(r,W)$ for the Grassmannian of $r$-dimensional quotients of $W$.

\begin{lemma}\label{simpext}
Let $\mathcal{Q}$, $\mathcal{R}$ denote the tautological quotient and sub-bundles on the Grassmannian $\mathbb{G}=\mathbb{G}(2l,W)$. Then
$$
\textnormal{Ext}^{\bullet}_S(J_{\underline{z},l},S)=\bigoplus_{\alpha\in A_l} H^{\binom{n}{2}-\binom{2l}{2}-\bullet}(\mathbb{G}, S_{\beta}\mathcal{R}\otimes S_{\alpha}\mathcal{Q} )^{\ast},
$$
where $\beta=(z^{(2)}_{2l+1}+n-1, \cdots , z^{(2)}_n+n-1)$ and 
\begin{equation}\label{aset}
A_l=\{ \alpha \in \mathbb{Z}^{2l}_{\text{dom}} \mid \text{$\alpha_{2i-1}=\alpha_{2i}$ for $1\leq i\leq l$, and $\alpha_1\leq z_1+n-2l$}\}.
\end{equation}
\end{lemma}

\begin{proof}
We define bundles
$$
\mathcal{S}^{\vee}=\textnormal{det}\left(\bigwedge^2 \mathcal{Q}^{\ast} \right)\otimes \Sym\left(\bigwedge^2 \mathcal{Q}^{\ast}\right),\;\text{and}\;\; \mathcal{V}=S_{\underline{z}^2}\mathcal{R}\otimes S_{\underline{z}^1}\mathcal{Q},
$$
where $\underline{z}^1=(z_1, \ldots ,z_{l})^{(2)}$ and $\underline{z}^2=(z_{l+1}, \ldots, z_m)^{(2)}$. By \cite[Lemma 2.5]{raicu2016local} we have
\begin{equation}\label{Bott}
\textnormal{Ext}^{\bullet}_S(J_{\underline{z},l},S)=H^{\binom{n}{2}-\binom{2l}{2}-\bullet}(\mathbb{G}, \mathcal{S}^{\vee}\otimes \mathcal{V})^{\ast}\otimes \textnormal{det}\left(\bigwedge^2 W^{\ast}\right).
\end{equation}
For ease of notation set $z_1=d$. We make a few observations:
$$
\textbf{(1)}\; \textnormal{det}\left(\bigwedge^2 W^{\ast}\right)=\textnormal{det}(W^{\ast})^{\otimes (n-1)}, \;\;\; \textbf{(2)}\;\textnormal{det}(W)\otimes \mathcal{O}_{\mathbb{G}}=\textnormal{det}\mathcal{Q}\otimes \textnormal{det}\mathcal{R}, \;\;\;\textbf{(3)}\;\textnormal{det}\left(\bigwedge^2 \mathcal{Q}^{\ast}\right)=\textnormal{det}(\mathcal{Q})^{\otimes (-2l+1)}.
$$
Then from (1),(2),(3), and (\ref{Bott}) we have that
\begin{align*}
\textnormal{Ext}^{\bullet}_S(J_{\underline{z},l},S) & = H^{\binom{n}{2}-\binom{2l}{2}-\bullet}\left(\mathbb{G}, \textnormal{det}(\mathcal{Q})^{\otimes (d+n-2l)}\otimes \Sym\left(\bigwedge^2 \mathcal{Q}^{\ast}\right)\otimes S_{\underline{z}^2}\mathcal{R}\otimes \textnormal{det}(\mathcal{R})^{\otimes (n-1)}   \right)^{\ast}\\
&=  \bigoplus_{\underline{y}\in \mathcal{P}(l)} H^{\binom{n}{2}-\binom{l}{2}-\bullet}\left(\mathbb{G}, S_{\beta}\mathcal{R}\otimes S_{\underline{y}^{(2)}}\mathcal{Q}^{\ast}\otimes \textnormal{det}(\mathcal{Q})^{\otimes (d+n-2l)}\right)^{\ast}.
\end{align*}
The result then follows from the plethysm identity (\ref{cauchy}) applied to $\Sym(\bw^2 \mathcal{Q}^{\ast})$.
\end{proof}

\noindent For $0\leq l \leq m$ and $\underline{z}\in \mathcal{P}(m)$, we define
\begin{equation}\label{ts}
\mathcal{T}_l(\underline{z})=\{ \underline{t}=(l=t_1\geq t_2 \geq \cdots \geq t_{n-2l})\in \mathbb{Z}^{n-2l}_{\geq 0} \mid z^{(2)}_{2l+i}-z^{(2)}_{2l+i+1}\geq 2t_i-2t_{i+1}\textnormal{ for }1\leq i \leq n-2l-1 \}.
\end{equation}

\begin{remark} Since $\underline{z}^{(2)}\in \mathcal{P}_e(n)$, it follows that $z_{2i}^{(2)}=z_{2i-1}^{(2)}$ for all $i=1,\cdots ,m$. Therefore, if $\underline{t}\in \mathcal{T}_l(\underline{z})$, then $t_{2i}=t_{2i-1}$ for all $i=1,\cdots , m-l$. However, this does not imply that $\underline{t}\in \mathcal{P}_e(n-2l)$ in the case when $n$ is odd, since $t_{n-2l}$ may not be zero. 
\end{remark}

We are now ready to state our Ext computations for the modules $J_{\underline{z},l}$:

\begin{theorem}\label{extjxp}
Fix $0\leq l \leq m-1$ and assume that $\underline{z}\in \mathcal{P}(m)$ with $z_1=\cdots =z_{l+1}$. For $\underline{t}\in \mathcal{T}_l(\underline{z})$ we consider the set $W(\underline{z},l, \underline{t})$ of dominant weights $\lambda\in \mathbb{Z}^n_{\text{dom}}$ satisfying
\begin{equation}\label{weightset}
\begin{cases} 
      \lambda_{2l+i-2t_i}=z^{(2)}_{2l+i}+n-1-2t_i  & \text{for}\;\;\; i=1,\cdots, n-2l,\\
      \lambda_{2i}=\lambda_{2i-1} & \text{for}\;\;\; 0<2i<n-2t_{n-2l},\\
      \lambda_{n-2i}=\lambda_{n-2i-1} & \text{for}\;\;\; 0\leq i \leq t_{n-2l}-1.
   \end{cases}
\end{equation}
Then for all $j\geq 0$
\begin{equation}\label{extjxp1}
\textnormal{Ext}^j_S(J_{\underline{z},l},S)=\bigoplus_{\substack{\underline{t}\in \mathcal{T}_l(\underline{z})\\ \binom{n}{2}-\binom{2l}{2}-2\sum_{i=1}^{n-2l}t_i=j \\\lambda\in W(\underline{z},l, \underline{t})}} S_{\lambda} W^{\ast},
\end{equation}
where $S_{\lambda}W^{\ast}$ appears in degree $-|\lambda|/2$.
\end{theorem}

\begin{proof}
In the notation of Lemma \ref{simpext}, we have that
\begin{equation}\label{extbott}
\textnormal{Ext}^{\bullet}_S(J_{\underline{z},l},S)=\bigoplus_{\alpha\in A_l} H^{\binom{n}{2}-\binom{2l}{2}-\bullet}(\mathbb{G},S_{\beta}\mathcal{R}\otimes S_{\alpha}\mathcal{Q})^{\ast}. 
\end{equation}
Write $\gamma=(\alpha_1, \cdots , \alpha_{2l}, \beta_1, \cdots , \beta_{n-2l})$ for $\alpha\in A_l$. Let $\rho=(n-1,n-2, \cdots ,0)$ and consider $\gamma+\rho=(\gamma_1+n-1, \gamma_2+n-2, \cdots , \gamma_n)$. We write $\bar{\gamma}=\textnormal{sort}(\gamma+\rho)$ for the sequence obtained by arranging the entries of $\gamma+\rho$ in non-increasing order. If we write $q$ for the number of pairs $(x,y)$ with $1\leq x<y\leq n$ and $\gamma_x-x<\gamma_y-y$, then Bott's Theorem for Grassmannians \cite[Corollary 4.1.9]{weyman2003cohomology} yields
\begin{equation}\label{bottform}
H^{\bullet}(\mathbb{G},S_{\beta}\mathcal{R}\otimes S_{\alpha}\mathcal{Q})^{\ast}= 
 \begin{cases} 
      S_{\bar{\gamma}-\rho}W^{\ast} &\textnormal{ if $\gamma+\rho$ has distinct entries and $\bullet=q$;}\\
      0 &\textnormal{ otherwise.}
   \end{cases}
\end{equation}
Suppose that $\gamma+\rho$ has distinct entries. We start by showing that $\bar{\gamma}-\rho\in W(\underline{z},l,\underline{t})$ for some $\underline{t}\in \mathcal{T}_l(\underline{z})$. For ease of notation, write $\lambda=\bar{\gamma}-\rho$. Let $\sigma$ denote the permutation of $\{1,\cdots ,n\}$ that sorts $\gamma+\rho$. This permutation is unique, as $\gamma+\rho$ has distinct entries. Notice that since $\alpha_1\geq \cdots \geq \alpha_{2l}$ and $\beta_1\geq \cdots \geq \beta_{n-2l}$ it follows that $\sigma(1)<\cdots <\sigma(2l)$ and $\sigma(2l+1)<\cdots <\sigma(n)$. Thus, $\lambda$ is determined uniquely by the partition $\underline{u}\in \mathcal{P}(n-2l)$ defined by
$$
u_i=2l+i-\sigma(2l+i), \textnormal{ for $i=1,\cdots, n-2l$}.
$$
We may view the sorting process as moving $\beta_1+\rho_{2l+1}$ left by $u_1$ spaces, then moving $\beta_2+\rho_{2l+2}$ left by $u_2$ spaces, and so on. Note that the condition $\gamma_x-x<\gamma_y-y$ is equivalent to $\bar{\gamma}_{\sigma(x)}<\bar{\gamma}_{\sigma(y)}$. If $\gamma+\rho$ has distinct entries, we get that
$$
q=\#\{ (x,y)\mid x<y, \sigma(x)>\sigma(y)\}=|\underline{u}|.
$$
Our proof is based on two claims which we explain at the end:

\underline{Claim 1:} The integer $u_1$ is maximal, i.e. $\sigma(2l+1)=1$ and $u_1=2l$.

\underline{Claim 2:} $u_i$ is even for all $i=1,\cdots ,n-2l$, and setting $t_i=u_i/2$ for all $i$ gives $\underline{t}\in \mathcal{T}_l(\underline{z})$.

We now explain the proof that $\lambda\in W(\underline{z},l,\underline{t})$ based on Claims 1 and 2. For ease of notation, set $\underline{w}=\underline{z}^{(2)}$. Notice that
\begin{equation}\label{termsoflam}
\lambda_{2l+i-u_i}=\beta_i-u_i=w_{2l+i}+n-1-u_i,
\end{equation}
as needed to confirm that $\lambda$ satisfies condition one of (\ref{weightset}). To see that $\lambda$ satisfies the second and third condition of (\ref{weightset}), assume first that $n$ is even. In this case, we need to show that $\lambda_{2i}=\lambda_{2i-1}$ for $i=1, \cdots ,m$. If $2i=2l+j-u_j$ for some $1\leq j \leq n-2l$, then $j$ is even by Claim 2. Since $w_{2l+j}=w_{2l+j-1}$, it follows from Claim 2 that $u_j=u_{j-1}$. Thus,
$$
\lambda_{2i}=\lambda_{2l+j-u_j}=w_{2l+j}+n-1-u_j= w_{2l+j-1}+n-1-u_{j-1}=\lambda_{2l+j-1-u_{j-1}}=\lambda_{2i-1},
$$
as needed. Fix $1\leq i\leq m$ and suppose that there does not exist $j=1,\cdots ,n-2l$ with $2i=2l+j-u_j$. Then there exists $1\leq j \leq l$ such that $\lambda_{2i}$ (resp. $\lambda_{2i-1}$) is obtained by moving $\alpha_{2j}+n-2j$ (resp. $\alpha_{2j-1}+n-2j+1$) $2i-2j$ spaces to the right when sorting $\gamma+\rho$. Thus,
$$
\lambda_{2i}=(\alpha_{2j}+n-2j)-(n-2i)=\alpha_{2j}+2i-2j=\alpha_{2j-1}+2i-2j=(\alpha_{2j-1}+n-2j+1)-(n-2i+1)=\lambda_{2i-1}.
$$
Setting $t_i=u_i/2$ for $i=1,\cdots, n-2l$ we see that $\lambda\in W(\underline{z},l, \underline{t})$, where $\underline{t}\in \mathcal{T}_l(\underline{z})$, as required. If $n$ is odd, the proof is similar except it may not be true that $\lambda_{n-2t_{n-2l}}=\lambda_{n-2t_{n-2l}+1}$.

For the reverse containment, assume that $\lambda\in W(\underline{z},l, \underline{t})$, where $\underline{t}\in \mathcal{T}_l(\underline{z})$. Its clear that a suitable weight $\alpha$ can be constructed from $\lambda$ such that $\lambda=\bar{\gamma}-\rho$.

We now prove the claims.

\underline{Proof of Claim 1:} Let $\alpha\in A_l$. If $l=0$ the claim is clear, so suppose $l\geq 1$. Since $\gamma+\rho$ has distinct entries, it suffices to show $\alpha_2+\rho_2\leq \beta_1+\rho_{2l+1}$. By (\ref{aset}) we have
$$
\alpha_2+\rho_2 \leq w_1+n-1+n-2l-1=\beta_1+n-2l-1=\beta_1+\rho_{2l+1},
$$
as desired.

\underline{Proof of Claim 2:} Since $\alpha_{2i}=\alpha_{2i-1}$ for $i=1,\cdots ,l$, it follows that $(\gamma+\rho)_{2i}=(\gamma+\rho)_{2i-1}-1$ for $i=1,\cdots, l$. Thus, it cannot be the case that $(\gamma+\rho)_{2i-1}-1>(\gamma+\rho)_j>(\gamma+\rho)_{2i}$ for some $i\leq l$ and $j>2l$. Therefore, $\beta_i+\rho_{2l+i}$ gets shifted left by an even amount while sorting, for all $i>l$, i.e. $u_j$ even for all $j=1,\cdots ,n-2l$. We only need to show that $w_{2l+i}-w_{2l+i+1}\geq u_i-u_{i+1}$ for all $i=1,\cdots ,n-2l-1$. Since $\lambda$ is dominant we have $\lambda_{2l+i-u_i}\geq \lambda_{2l+i+1-u_{i+1}}$ for all $i=1,\cdots, n-2l-1$, so that by (\ref{termsoflam}), $\beta_i-u_i\geq \beta_{i+1}-u_{i+1}$, which is equivalent to the desired inequality by the definition of $\beta$.
\end{proof}

Given $\underline{z}\in \mathcal{P}(m)$ and $\underline{t}\in \mathcal{T}_l(\underline{z})$, we define
\begin{equation}
f_{l}(\underline{z},\underline{t})=\sum_{i=1}^{n-2l-1}t_{i+1}((z^{(2)}_{2l+i}-z^{(2)}_{2l+i+1})-(2t_i-2t_{i+1})).
\end{equation}

\begin{theorem}\label{regJ}
For $0\leq l \leq m-1$ and $\underline{z}\in \mathcal{P}(m)$ satisfying $z_1=\cdots =z_{l+1}$, we have that
\begin{equation}\label{regularityformulaJ}
\textnormal{reg}\left( J_{\underline{z},l}\right)=\max_{\underline{t}\in \mathcal{T}_l(\underline{z})} \left(\sum_{i=l+1}^{m} z_i+l(z_1-1)+ |\underline{t}|-f_l(\underline{z},\underline{t}) \right).
\end{equation}
\end{theorem}

\begin{proof}
By (\ref{compreg}) and Theorem \ref{extjxp}, $\textnormal{reg}(J_{\underline{z},l})$ is given by
$$
\textnormal{reg}(J_{\underline{z},l})=\max \left\{-r-j\mid \textnormal{Ext}^j_S(J_{\underline{z},l},S)_r\neq 0\right\}= \max_{\substack{ \underline{t}\in \mathcal{T}_l(\underline{z})\\ \lambda\in W(\underline{z},l, \underline{t}) }} \left( \frac{|\lambda|}{2}-\binom{n}{2}+\binom{2l}{2}+2|\underline{t}| \right).
$$
Given $\underline{t}\in \mathcal{T}_l(\underline{z})$, let $\lambda$ be the largest weight in $W(\underline{z},l,\underline{t})$ with respect to $|\lambda|$. Notice that 
$$
|\lambda|=\sum_{i=1}^{n-2l-1} (2t_i-2t_{i+1}+1)(z^{(2)}_{2l+i}+n-1-2t_i)+(2t_{n-2l}+1)(z^{(2)}_n+n-1-2t_{n-2l}).
$$
Therefore, the proof is completed by verifying the following equality, which we leave to the interested reader.
$$
\frac{|\lambda|}{2}-\binom{n}{2}+\binom{2l}{2}+2|\underline{t}|=\sum_{i=l+1}^m z_i+l(z_1-1)+|\underline{t}|-f_l(\underline{z},\underline{t}).
$$
\end{proof}

\subsection{Discussion of Theorem F}\label{mainsub}
In this section we make precise the statement of Theorem F. As the proof is virtually identical to the proof of \cite[Theorem 3.3]{raicu2016regularity}, we refer the reader to that paper for many details. Let $I\subseteq S$ be a GL-invariant ideal. We continue to write $m=\lfloor n/2 \rfloor$.

\begin{definition}\label{subquotset}
\textnormal{For a finite subset $\mathcal{X}\subset \mathcal{P}(m)$, we define $\mathcal{Z}(\mathcal{X})$ to be the set consisting of pairs $(\underline{z},l)$ where $\underline{z}\in \mathcal{P}(m)$ and $0\leq l \leq m-1$ are such that if we write $c=z_1$ then the following hold:}
\begin{enumerate}
\item \textnormal{There exists a partition $\underline{x}\in \mathcal{X}$ such that $\underline{x}(c)\leq \underline{z}$ and $x'_{c+1}\leq l+1$}.
\item \textnormal{For every partition $\underline{x}\in \mathcal{X}$ satisfying (1) we have $x'_{c+1}=l+1$.}
\end{enumerate}
\end{definition}

\noindent We may now describe the sets $\mathcal{M}(I)$ in the statement of Theorem F. Given $\mathcal{X}\subset \mathcal{P}(m)$, we have $\mathcal{M}(I_{\mathcal{X}})$ is the set of $J_{\underline{z},l}$, where $(\underline{z},l)\in \mathcal{Z}(\mathcal{X})$. By an argument identical to the proof of \cite[Corollary 3.7]{raicu2016regularity}, it follows that $S/I_{\mathcal{X}}$ has a $\tn{GL}$-equivariant filtration by $S$-modules:
\begin{equation}\label{mainfilt}
S/I_{\mathcal{X}}=M_0\supsetneq M_1\supsetneq \cdots \supsetneq M_{t-1}\supsetneq M_t=0,
\end{equation}
where the quotients $M_i/M_{i+1}$ are all among $J_{\underline{z},l}$ for $(\underline{z},l)\in \mathcal{Z}(\mathcal{X})$. Further, for all $(\underline{z},l)\in \mathcal{Z}(\mathcal{X})$, the module $J_{\underline{z},l}$ appears as a quotient in the filtration (\ref{mainfilt}) with multiplicity one. This filtration yields short exact sequences
\begin{equation}\label{splinter}
0\longrightarrow M_{i+1} \longrightarrow M_i\longrightarrow M_i/M_{i+1}\longrightarrow 0,
\end{equation}
and each of these short exact sequences induces a long exact sequence of $\tn{Ext}^{\bullet}_S(-,S)$. Part (\ref{extformula}) of Theorem F follows from the fact that the connecting homomorphisms in these long exact sequences are zero for all $i=0,\cdots, t-1$. In other words, for all $j\geq 0$ and all $i=0,\cdots, t-1$, there are exact sequences
$$
0\longrightarrow \tn{Ext}^j_S(M_i/M_{i+1},S)\longrightarrow \tn{Ext}^j_S(M_i,S)\longrightarrow \tn{Ext}^j_S(M_{i+1},S)\longrightarrow 0,
$$
inducing the isomorphism (\ref{extformula}) as representations of $\tn{GL}$, but possibly not as $S$-modules.

\begin{remark}\label{subquotsilambda}
If $\mathcal{X}=\{\underline{x}\}$ is a singleton, then 
\begin{equation}
\mathcal{Z}(\mathcal{X})=\left\{ (\underline{z},x_{c+1}'-1) \mid 0\leq c\leq x_1-1, z_1=\cdots = z_{x_{c+1}'}=c \textnormal{ and } \underline{x}(c)\leq \underline{z}    \right\}.
\end{equation}
This will be used when computing the regularity of the ideals $I_{\underline{x}}$.
\end{remark}

\begin{remark}\label{pfaffisJ}
Recall that the Pfaffian ideal $I_{2k}$ is the basic invariant ideal $I_{\underline{x}}$, where $\underline{x}=(1^k)$. By the previous remark, it follows that $\mathcal{Z}(\underline{x})=\{(\underline{0},k-1)\}$. This means that $S/I_{2k}=J_{\underline{0},k-1}$ for all $1\leq k\leq m$, so the filtration (\ref{mainfilt}) is trivial in this case. In Section \ref{proofReg} we explicitly describe $\mathcal{Z}(\mathcal{X})$ when $\mathcal{X}=\mathcal{X}_k^d$, $\mathcal{X}_k^{(d)}$, $\mathcal{X}_k^{d:1}$ is the indexing set for a power, symbolic power, or saturation of a power of a Pfaffian ideal, respectively. 
\end{remark}

\begin{remark}
Let $R=\Sym(\Sym^2 W)$ be the ring of polynomial functions on the space of $n\times n$ symmetric matrices. In \cite[Lemma 2.3]{raicu2016local} the authors introduce analogues of the $J_{\underline{z},l}$'s over $R$, referred to by $J_{\underline{z},l}^{symm}$. The proof of \cite[Corollary 3.7]{raicu2016regularity} may be adapted to this situation, i.e. for all $\tn{GL}(W)$-invariant ideals $I\subseteq R$, there exists a filtration like (\ref{mainfilt}), where the subquotients are of the form $J_{\underline{z},l}^{symm}$. However, the long exact sequences of Ext arising from the short exact sequences (\ref{splinter}) do not behave as nicely as in the generic or skew-symmetric cases. In particular, the connecting homomorphisms are sometimes non-zero. For more information on this example, see \cite[Remark 2.7]{raicu2016local}.
\end{remark}

The proof of Theorem F is almost identical to the proof of \cite[Theorem 3.3]{raicu2016regularity}, with the two exceptions being Raicu's Lemma 3.8 and (part of) Lemma 3.11. Proving the analogues of these requires the computations from our Theorem \ref{extjxp}. We prove our version of these lemmas below. We first recall notation from \cite{raicu2016regularity}: to every $(\underline{z},l)$ with $\underline{z}\in \mathcal{P}(m)$ and $z_1=\cdots =z_{l+1}$, we associate the collection of rectangular partitions
\begin{equation}\label{rect}
\mathcal{Y}_{\underline{z},l}=\left\{ (z_1+1)^{l+1}\right\}\cup \left\{ (z_i+1)^{i} \mid i>l+1\textnormal{ and} \;\; z_{i-1}>z_i \right\}.
\end{equation}

\begin{lemma}\label{nomaps}
Let $\mathcal{Y}=\mathcal{Y}_{\underline{z},l}\cup \{\underline{z}\}$. If $(\underline{y},k)\in \mathcal{Z}(\mathcal{Y})$, then the $S$-modules $\textnormal{Ext}^j_S(J_{\underline{z},l},S)$ and $\textnormal{Ext}^{j+1}_S(J_{\underline{y},k},S)$ share no irreducible $\tn{GL}$-representations for all $j\geq 0$.
\end{lemma}

\begin{proof}
Write $z_1=c$ and let $(\underline{y},k)\in \mathcal{Z}(\mathcal{Y})$ with $y_1=d$. Then by (\ref{subquotset}) there's $\underline{x}\in \mathcal{Y}$ such that $\underline{y}\geq \underline{x}(d)$ and $x'_{d+1}=k+1$. By (\ref{rect})  the nonzero columns of $\underline{x}$ have size at least $l+1$, so that $k+1\geq l+1$, i.e. $k\geq l$. Since $x_1\geq d+1$, and since $x_1\leq c+1$, we conclude that $d\leq c$. 

Now assume that $S_{\lambda}W^{\ast}$ is a subrepresentation of both $\textnormal{Ext}^j_S(J_{\underline{z},l},S)$ and $\textnormal{Ext}^{j+1}_S(J_{\underline{y},k},S)$. Then by (\ref{extjxp1}) it follows that $\lambda_1=c+n-1-2l=d+n-1-2k$. Since $k\geq l $ and $d\leq c$, we see that $k=l$ and $c=d$. Also by (\ref{extjxp1}), $\lambda$ is parametrized by partitions $\underline{t}\in \mathcal{T}_l(\underline{z})$ and $\underline{s}\in \mathcal{T}_l(\underline{y})$ such that
\begin{equation}\label{contradiction}
\binom{n}{2}-\binom{2l}{2}-2\sum_{i=1}^{n-2l}t_i=j\;\;\;\textnormal{and} \;\;\;\binom{n}{2}-\binom{2l}{2}-2\sum_{i=1}^{n-2l}s_i=j+1
\end{equation}
This cannot happen because the left hand sides of both expressions in (\ref{contradiction}) have the same parity. We conclude that no such $\lambda$ exists, which completes the proof of the lemma.
\end{proof}

\begin{lemma}\label{lemmatoprove}
Let $0\leq l\leq m$ be and let $\underline{z}\in \mathcal{P}(m)$. If
$$
\mathcal{X}=\left\{ ((z_i+1)^{i}) \mid i>l+1 \right\}\subset \mathcal{P}(m),
$$
then the $S$-modules $\textnormal{Ext}^j_S(J_{\underline{z},l},S)$ and $\textnormal{Ext}^j_S(S/I_{\mathcal{X}},S)$ share no irreducible GL-representations for all $j\geq 0$.
\end{lemma}

\begin{proof}
Every nonzero column of $\underline{x}\in \mathcal{X}$ has size bigger than $l+1$, so it follows from part (2) of Definition \ref{subquotset} that if $(\underline{y},u)\in \mathcal{Z}(\mathcal{X})$ then $u+1>l+1$, i.e. $u>l$. Every $\underline{x}\in \mathcal{X}$ satisfies $x_1\leq c+1$, so if $(\underline{y},u)\in \mathcal{Z}(\mathcal{X})$ then $y_1\leq c$. If $S_{\lambda}W^{\ast}$ occurs as a subrepresentation of $\textnormal{Ext}^j_S(S/I_{\mathcal{X}},S)$ then there exists $(\underline{y},u)\in \mathcal{Z}(\mathcal{X})$ such that $S_{\lambda}W^{\ast}$ occurs inside $\textnormal{Ext}^j_S(J_{\underline{y},u},S)$. It follows from Theorem \ref{extjxp} that $\lambda_1=y_1+n-1-2u$. If $S_{\lambda}W^{\ast}$ also occurs in $\textnormal{Ext}^j_S(J_{\underline{z},l},S)$, by the same reasoning we have $\lambda_1=z_1+n-1-2l$.
Thus,
$$
z_1+n-1-2l=y_1+n-1-2u \iff z_1-2l=y_1-2u,
$$
and since $u>l$ and $y_1\leq c=z_1$, this is a contradiction.
\end{proof}

\section{Regularity of Basic Thickenings}\label{Section4}

As before, $S$ will denote the ring of polynomial functions on the space of $n\times n$ complex skew-symmetric matrices, and we write $m=\lfloor n/2 \rfloor$. In this section we use the Ext computations from the previous section to compute the regularity of the basic GL-invariant ideals in $S$ (for reminder of definition see (\ref{ilambda})). As a consequence, we determine the partitions $\underline{x}\in \mathcal{P}(m)$ for which $I_{\underline{x}}$ has linear minimal free resolution. We start by treating the case where $n$ is even, followed by the case when $n$ is odd.

\begin{lemma}\label{regilambda1}
If $n$ is even and $0\leq l \leq m$, then for an integer $c\geq 0$ we have
\begin{equation}
\max \left\{ \textnormal{reg}\left( J_{\underline{z},l}\right)\mid \underline{z}\in \mathcal{P}(m), z_1=\cdots =z_{l+1}=c \right\}= \textnormal{reg}\left( J_{(c^m),l}\right)=(n-2l-1)l+cm.
\end{equation}
\end{lemma}

\begin{proof}
By Theorem \ref{regJ}, if $\underline{z}\in \mathcal{P}(m)$ and $z_1=\cdots =z_{l+1}=c$, we have
\begin{equation}
\textnormal{reg}\left( J_{\underline{z},l}\right)=\max_{\underline{t}\in \mathcal{T}_l(\underline{z})} \left( \sum_{i=l+1}^m z_i +l(c-1)+|\underline{t}|-f_l(\underline{z},\underline{t}) \right).
\end{equation}
Note that $|\underline{t}|$ is maximized for $\underline{t}=(l^{n-2l})$ and $\sum_{i=l+1}^m z_i$ is maximized for $\underline{z}=(c^m)$. Since $t_i=t_{i+1}$ for all $i=1, \cdots , n-2l-1$ and $z_{l+i}=z_{l+i+1}$ for all $i=1, \cdots , m-l-1$, we see that $\underline{t}\in \mathcal{T}_l(\underline{z})$ and $f_l(\underline{z},\underline{t})=0$ for this choice of $\underline{z}$ and $\underline{t}$. Therefore, the result follows.
\end{proof}

\begin{theorem}\label{regularityIlambda}
If $n$ even and $\underline{x}\in \mathcal{P}(m)$, then
\begin{equation}
\textnormal{reg}\left(I_{\underline{x}}\right)= \max_{0\leq c\leq x_1-1} \left((n-2x_{c+1}'+1)(x_{c+1}'-1)+cm+1 \right).
\end{equation}
\end{theorem}

\begin{proof}
By part (\ref{extformula}) in Theorem F and Remark \ref{subquotsilambda}, we have
$$
\textnormal{reg}(S/I_{\underline{x}})= \max_{\substack{(\underline{z},l)\in \mathcal{Z}(\{\underline{x}\})   }} \left( \textnormal{reg} \left( J_{\underline{z},l} \right) \right) \overset{\textnormal{Lemma} \;\ref{regilambda1}}=\max_{0\leq c\leq x_1-1} \left((n-2x_{c+1}'+1)(x_{c+1}'-1)+cm \right).
$$
The result follows after recalling that $\textnormal{reg}(I_{\underline{x}})=\textnormal{reg}(S/I_{\underline{x}})+1$.
\end{proof}

For the following statement, note that by (\ref{highestweight}) the ideal $I_{\underline{x}}$ is generated in degree $|\underline{x}|$. 

\begin{corollary}\label{linearreseven}
If $n$ is even, then $I_{\underline{x}}$ has linear free resolution if and only if $\underline{x}'=(m^k,1^l)$, where $k\geq 0$ and $l\in \{0,1 \}$. In other words, when $n$ is even, basic invariant ideals have linear free resolution if and only if they are obtained from $S$ or the homogeneous maximal ideal by multiplying by a power of the $n\times n$ Pfaffian.
\end{corollary}

\begin{proof}
$\impliedby:$ Using Theorem \ref{regularityIlambda} it is easy to check that $\textnormal{reg}(I_{\underline{x}})=|\underline{x}|$ if $\underline{x}$ is of the form stated.

$\implies:$ Now suppose $\underline{x}\in \mathcal{P}(m)$ and $I_{\underline{x}}$ has linear free resolution. Note that if $m\leq 2$, there is nothing to show, so assume $m\geq 3$. We start by showing that $\underline{x}$ has at most one column of length one. Let $s=x_1$ and suppose for contradiction that $x_s'=\cdots =x_{s-t}'=1$ for $t\geq 1$. Under these constraints, the largest $\underline{x}$ can be is $\underline{x}'=(m^{s-t-1},1^{t+1})$, in which case $|\underline{x}|=(s-t-1)m+(t+1)$. By Theorem \ref{regularityIlambda}, we have that
$$
\textnormal{reg}(I_{\underline{x}})\geq (n-2x_s'+1)(x_s'-1)+(s-1)m+1=(s-1)m+1.
$$
Thus, it would be a contradiction if $m(s-1)+1>(s-t-1)m+(t+1)\geq |\underline{x}|$. Note that
$$
m(s-1)+1\leq (s-t-1)m+(t+1) \iff t(1-m)\geq 0 \iff t=0\textnormal{ or }m\leq 1,
$$
but $t\geq 1$ and $m\geq 3$, so the claim is proved.

Suppose for contradiction that $\underline{x}$ is not of the form stated, and let $c$ be largest such that $1<x_{c+1}'<m$. For ease of notation, set $x_{c+1}'=d$. By the previous paragraph, the largest $\underline{x}$ can be is $\underline{x}'=(m^c,d,1)$, in which case $|\underline{x}|=cm+d+1$. By Theorem \ref{regularityIlambda} we know that
$$
\textnormal{reg}(I_{\underline{x}})\geq (n-2d+1)(d-1)+cm+1.
$$
Thus, it would be a contradiction if $(n-2d+1)(d-1)+cm+1>cm+d+1\geq |\underline{x}|$. Note that 
$$
(n-2d+1)(d-1)+cm+1> cm+d+1\iff d(n-2d+2)-n-1>0.
$$
Since $2\leq d\leq m-1$ and $n\geq 6$, its easy to check that this inequality holds, completing the proof.
\end{proof}

We now compute the regularity of $I_{\underline{x}}$ when $n$ is odd. Recall the definition of $b:\mathbb{Z}^3\to \mathbb{Z}$ from (\ref{bfunction}).

\begin{lemma}\label{ilambdaregodd}
If $n$ is odd and $0\leq l \leq m-1$, then for an integer $c\geq 0$ we have
\begin{equation}\label{regformulalambdaodd}
\max \left\{ \textnormal{reg}\left( J_{\underline{z},l}\right)\mid \underline{z}\in \mathcal{P}(m), z_1=\cdots =z_{l+1}=c\right\}=\textnormal{reg}\left(J_{(c^m),l}\right)=b(l,n,c).
\end{equation}
\end{lemma}

\begin{proof}
We start by showing that $\textnormal{reg}(J_{(c^{m}),l})=b(l,n,c)$. By Theorem \ref{regJ} we have 
$$
\textnormal{reg}\left(J_{(c^m),l}\right)=\max_{\underline{t}\in \mathcal{T}_l(c^m)} \left( (m-l)c+l(c-1)+|\underline{t}|-f_l(c^m,\underline{t})\right).
$$
Since $\underline{t}\in \mathcal{T}_l(c^m)$ it follows that $t_1=\cdots =t_{n-2l-1}=l$, so that 
\begin{equation}\label{simplifiedreg}
\textnormal{reg}\left(J_{(c^m),l}\right)=\max_{\underline{t}\in \mathcal{T}_l(c^m)}\left( (m-l)c+l(c-1)+2l(m-l)+t_{n-2l}-t_{n-2l}(c-2l+2t_{n-2l})\right).
\end{equation}
Thus, we need to maximize 
\begin{equation}\label{functiontomaxodd}
t_{n-2l}-t_{n-2l}(c-2l+2t_{n-2l}),
\end{equation}
with respect to $t_{n-2l}$. If $c>2l$ then (\ref{functiontomaxodd}) is maximized when $t_{n-2l}=0$, since $t_{n-2l}\geq 0$. If $c\leq 2l$ then we find that (\ref{functiontomaxodd}) is maximized when $t_{n-2l}=(2l-c)/2$ if $c$ is even, and $t_{n-2l}=(2l-c+1)/2$ if $c$ is odd. Using (\ref{simplifiedreg}) we see that $\textnormal{reg}\left(J_{(c^m),l}\right)=b(l,n,c)$, as claimed.

We now show that if $\underline{w}\in \mathcal{P}(m)$ and $w_1=\cdots =w_{l+1}=c$, then $\textnormal{reg}\left(J_{\underline{w},l}\right)\leq \textnormal{reg}\left(J_{(c^m),l}\right)$, completing the proof. Let $\underline{w}$ be lexicographically maximal such that: $w_1=\cdots =w_{l+1}=c$ and 
\begin{equation}\label{lexmax}
\textnormal{reg}\left(J_{\underline{w},l}\right)=\max \left\{ \textnormal{reg}\left( J_{\underline{z},l}\right)\mid \underline{z}\in \mathcal{P}(m), z_1=\cdots =z_{l+1}=c\right\},
\end{equation}
and let $\underline{t}\in \mathcal{T}_l(\underline{w})$ be such that
$$
\textnormal{reg}\left(J_{\underline{w},l}\right)=\sum_{i=l+1}^m w_i+l(c-1)+|\underline{t}|-f_l(\underline{w},\underline{t}).
$$
Note that such a $\underline{t}$ exists by Theorem \ref{regJ}.  

First assume that $w^{(2)}_{2l+i}-w^{(2)}_{2l+i+1}=2t_i-2t_{i+1}$ for all $i=1,\cdots, n-2l-1$. By (\ref{ts}), we have that $w^{(2)}_{2l+i}-2t_i=c-2l$ for all $i=1,\cdots, n-2l-1$. Since $w^{(2)}_{n}=0$ it follows that $t_{n-2l}=(2l-c)/2$, and since $t_{n-2l}\geq 0$ we see that $2l\geq c$ and $c$ is even. As $f_l(\underline{w},\underline{t})=0$ in this case, we need to show that
$$
\sum_{i=l+1}^m w_i+l(c-1)+|\underline{t}|\leq b(l,n,c).
$$
Since $|\underline{w}|$ is maximized when $\underline{w}=(c^m)$ and $|\underline{t}|$ is maximized when $\underline{t}=(l^{n-2l-1},(2l-c)/2)$, the result follows.

Now assume that there exists $1\leq j \leq n-2l-1$ such that $w^{(2)}_{2l+j}-w^{(2)}_{2l+j+1}>2t_j-2t_{j+1}$, and notice that this forces $j$ to be even. Let $j$ be the minimal index so that this inequality holds. If $j<n-2l-1$ then we define $\underline{x}\in \mathcal{P}(m)$ via $x^{(2)}_{2l+j+1}=x^{(2)}_{2l+j+2}=w^{(2)}_{2l+j+1}+1$, and $x^{(2)}_i=w^{(2)}_i$ otherwise. Notice that $\underline{t}\in \mathcal{T}_l(\underline{x})$. Then $x_1=\cdots=x_{l+1}=c$ and by Theorem \ref{regJ} we know that
$$
\textnormal{reg}\left(J_{\underline{x},l}\right) \geq \sum_{i=l+1}^m x_i+l(c-1)+|\underline{t}|-f_l(\underline{x},\underline{t})= \textnormal{reg}\left(J_{\underline{w},l}\right)+1+t_{j+1}-t_{j+3}.
$$
Since $t_{j+1}\geq t_{j+3}$, this contradicts the fact that $\underline{w}$ is lexicographically maximal such that (\ref{lexmax}) holds. 

Now suppose that $j=n-2l-1$.

\textbf{Case 1:} $t_{n-2l}=0$. If $t_{n-2l}=0$ then $w^{(2)}_{n-1}>2t_{n-2l-1}$ by (\ref{ts}), and $w^{(2)}_{2l+i}-w^{(2)}_{2l+i+1}\geq 2t_i-2t_{i+1}$ for $i=1,\cdots ,n-2l-1$ implies that $w^{(2)}_{2l+i}-2t_i\geq w^{(2)}_{2l+i+1}-2t_{i+1}$ for all $i=1,\cdots , n-2l-1$. Thus, $w^{(2)}_{2l+1}-2t_1\geq w^{(2)}_{n-1}-2t_{n-2l-1}>0$, so that $c=w^{(2)}_{2l+1}>2t_1=2l$. By Theorem \ref{regJ} we have that
$$
\textnormal{reg}\left(J_{\underline{w},l}\right)=\sum_{i=l+1}^{m} w_i+l(c-1)+|\underline{t}|-f_l(\underline{w},\underline{t})\leq (m-l)c+l(c-1)+l(n-2l-1),
$$
which is equal to $b(l,n,c)$ since $c>2l$, as needed.

\textbf{Case 2:} $t_{n-2l}\geq 1$. If $w^{(2)}_{n-1}=2t_{n-2l-1}-2t_{n-2l}+1$ then by Theorem \ref{regJ} we have
$$
\textnormal{reg}\left(J_{\underline{w},l}\right)=\sum_{i=l+1}^{m} w_i+l(c-1)+|\underline{t}|-t_{n-2l}\leq (m-l)c+l(c-1)+(n-2l-1)l.
$$
If $c>2l$, or $c\leq 2l$ and $c$ is odd, the right side of this inequality is equal to $b(l,n,c)$. If $c\leq 2l$ and $c$ is even, then $l(c-1)\leq c(l-1/2)$, so the right hand side of the above inequality is less than or equal to $b(l,n,c)$.

Now assume that $w^{(2)}_{n-1}>2t_{n-2l-1}-2t_{n-2l}+1$ and let $\underline{s}\in \mathcal{P}(n-2l)$ be defined by $s_{n-2l}=t_{n-2l}-1$ and $s_i=t_i$ otherwise. Its easy to see that $\underline{s}\in \mathcal{T}_l(\underline{w})$ and by Theorem \ref{regJ} we have that
\begin{align*}
\textnormal{reg}\left(J_{\underline{w},l}\right) & \geq \sum_{i=l+1}^m w_i+l(c-1)+|\underline{s}|-f_l(\underline{w},\underline{s})\\
& = \sum_{i=l+1}^m w_i+l(c-1)+|\underline{t}|-f_l(\underline{w},\underline{t})+4t_{n-2l}+w^{(2)}_{n-1}-2t_{n-2l-1}-3,
\end{align*}
which is greater than $\textnormal{reg}\left(J_{\underline{w},l}\right)$ since $t_{n-2l}\geq 1$ and $w^{(2)}_{n-1}>2t_{n-2l-1}-2t_{n-2l}+1$, yielding a contradiction.
\end{proof}

\begin{theorem}\label{thmIlambdaodd}
If $n$ is odd and $\underline{x}\in \mathcal{P}(m)$, then
\begin{equation}\label{thmformulaodd}
\textnormal{reg}\left( I_{\underline{x}}\right)=\max_{0\leq c\leq x_1-1} \left( b(x_{c+1}'-1,n,c)+1\right).
\end{equation}
\end{theorem}

\begin{proof}
The proof is analogous to that of Theorem \ref{regularityIlambda} and uses part (\ref{extformula}) in Theorem F and Lemma \ref{ilambdaregodd}.
\end{proof}

\begin{corollary}\label{linearreslambdaodd}
If $n$ is odd, then $I_{\underline{x}}$ has linear free resolution if and only if one of the following holds:
\begin{enumerate}
\item $\underline{x}'=(m^k)$ for $k\geq n-2$ or $k$ is even and $k\leq n-3$
\item $\underline{x}'=(m^k,1)$ for $k\geq n-5$ or $k$ is even and $k\leq n-7$.
\end{enumerate}
\end{corollary}

\begin{proof}
$\impliedby$: Using Theorem \ref{thmIlambdaodd} it is easy to check that $\textnormal{reg}(I_{\underline{x}})=|\underline{x}|$ if $\underline{x}$ is of the form stated.

$\implies$: Suppose now that $\underline{x}\in \mathcal{P}(m)$ and $I_{\underline{x}}$ has a linear free resolution.

\textbf{Case 1:} $x_i'=1$ for some $1\leq i \leq x_1$. In this case, we claim that $\underline{x}$ has a unique column of length $1$, i.e. $x_{x_1}'=1$ and $x_{x_1-1}'>1$. The proof of this fact follows by similar reasoning as in the proof of Corollary \ref{linearreseven} and is left to the interested reader.

Now suppose for contradiction that $\underline{x}$ has a column of length not equal to $1$ or $m$, and let $c$ be largest such that $1<x_{c+1}'<m$. In this case we must have $n\geq 7$. For ease of notation, set $d=x_{c+1}'$. Then the largest $\underline{x}$ can be is $\underline{x}'=(m^c,d,1)$, in which case $|\underline{x}|=cm+d+1$. We know that $\textnormal{reg}(I_{\underline{x}}) \geq b(d-1,n,c)+1$. Since $I_{\underline{x}}$ has linear minimal free resolution, it follows that $\textnormal{reg}(I_{\underline{x}})=|\underline{x}|$. Thus, it would be a contradiction if $b(d-1,n,c)+1>cm+d+1\geq |\underline{x}|$. Since $n\geq 7$, the formula (\ref{bfunction}) gives the desired contradiction. Therefore $\underline{x}'=(m^k,1)$ for some $k\geq 0$.

We only need to verify that if $k$ is odd and $k\leq n-6$ then $\textnormal{reg}(I_{\underline{x}})> km+1$. It suffices to show that if $k$ is odd and $k\leq n-6$ then $b(m-1,n,k-1)>km$. In this case, $b(m-1,n,k-1)=k(m-1/2)+m-3/2>km$.

\textbf{Case 2:} $\underline{x}$ has no column of length $1$. In this case we have that $\underline{x}$ has no columns of length $d<m$. Indeed, the proof of this fact is identical to the proof in the previous case. We conclude that $\underline{x}'=(m^k)$ for some $k\geq 1$. We need to show that if $k$ is odd and $k\leq n-4$ that $\textnormal{reg}(I_{\underline{x}})>km$. It suffices to show that if $k$ is odd and $k\leq n-4$ then $b(m-2,n,k-1)+1>km$. In this case, $b(m-1,n,k-1)+1=k(m-1/2)+m-1/2>km$.
\end{proof}

\section{Regularity of Powers of Ideals of Pfaffians}\label{Section5}

In this section we solve an optimization problem to prove Theorem A. The proof is concluded in Section \ref{proofReg}, where we also discuss the proof of Theorem G and applications to local cohomology.

\subsection{An optimization problem}\label{optimization}

The main result of this subsection is Proposition \ref{optimaleven} below. For $q$ even and positive integers $k$, $n$ with $0\leq q < 2k \leq n$, we consider the following set of pairs of partitions:
\begin{equation}\label{YU}
\mathcal{YU}(q,k,n,d)=\left\{ (\underline{y},\underline{u})\in \mathcal{P}_e(n-q)\times \mathcal{P}(n-q) \mid  \substack{  |\underline{y}|\leq d(2k-q)-2,\; |\underline{y}|-2y_1\geq d(2k-q-2)\\ \textnormal{$u_i$ even},\; u_1=q,\; y_i-y_{i+1}\geq u_i-u_{i+1}\textnormal{ for $i=1, \cdots , n-q-1$}               }  \right\}.
\end{equation}
For $(\underline{y},\underline{u})\in \mathcal{YU}(q,k,n,d)$ we define
\begin{equation}\label{gfunction}
g_{q,k,n,d}(\underline{y},\underline{u})=\frac{|\underline{y}|}{2}+\frac{|\underline{u}|}{2}+\frac{q(y_1-1)}{2}-\sum_{i=1}^{n-q-1} \frac{u_{i+1}}{2}((y_i-y_{i+1})-(u_i-u_{i+1})),
\end{equation}
and let
\begin{equation}\label{Rfunction}
R_{q,k,n,d}= \max_{(\underline{y},\underline{u})\in \mathcal{YU}(q,k,n,d)} g_{q,k,n,d}(\underline{y},\underline{u}),
\end{equation}
with the convention that $R_{q,k,n,d}=-\infty$ when $\mathcal{YU}(q,k,n,d)$ is empty. 

\begin{proposition}\label{optimaleven}
If one of the following holds:
\begin{enumerate}
\item $n$ is even, $q$ is even, $0\leq q<2k\leq n-2$, and $d\geq n-2$
\item $n$ is odd, $q$ is even, $0\leq q<2k\leq n-2$, and $d\geq n-3$
\item $n$ is even, $2k=n$, $q=2k-2$, and $d\geq 1$
\item $n$ is odd, $2k=n-1$, $q=2k-2$, and $d\geq n-3$
\end{enumerate}
then
$$
R_{q,k,n,d}=dk-1+q\left(k-\frac{q}{2}-1\right).
$$
\end{proposition}

The proof of this statement is the content of the remainder of the subsection. Consider $q$ even and positive integers $k$, $n$ with $0\leq q < 2k \leq n$. If $(\underline{y},\underline{u})\in \mathcal{YU}(q,k,n,d)$ then the first entry of $\underline{y}$ satisfies:
\begin{equation}\label{firsttermy}
y_1\leq d-1.
\end{equation}

We now verify Proposition \ref{optimaleven} in case (3). In Lemma \ref{maximalpfaffodd} we verify the proposition in case (4).

\begin{lemma}\label{maximalpfaffeven}
If $n$ is even we have that $R_{n-2,\frac{n}{2},n,d}=\frac{nd}{2}-1$ for all $d\geq 1$, and $R_{q,\frac{n}{2},n,d}=-\infty$ for $q\leq n-4$.
\end{lemma}

\begin{proof}
Let $q=n-2$ and $k=n/2$. The partitions $\underline{y}$, $\underline{u}$ satisfying the conditions in (\ref{YU}) have two equal parts, $\underline{y}=(y_1,y_1)$, $\underline{u}=(u_1,u_1)$, and they satisfy the conditions $y_1\leq d-1$ and $u_1=q=n-2$. Then
$$
g_{q,k,n,d}(\underline{y},\underline{u})=y_1+(n-2)+\frac{(n-2)(y_1-1)}{2},
$$
which is maximized when $y_1=d-1$. It follows that $R_{n-2,\frac{n}{2},n,d}=\frac{nd}{2}-1$, as claimed.

Assume now that $k=n/2$ and $q\leq n-4$. Then $(n-q-2)y_3 \geq |\underline{y}|-2y_1\geq d(2k-q-2)=d(n-q-2)$, so $y_3\geq d$, contradicting (\ref{firsttermy}). Therefore, $\mathcal{YU}(q,\frac{n}{2},n,d)$ is empty for $q\leq n-4$.
\end{proof}

\begin{lemma}\label{maximalpfaffodd}
If $n$ is odd we have that
$$
R_{n-3,\frac{n-1}{2},n,d}=
\begin{cases}
\frac{d(n-1)}{2}+\frac{1}{2}(n-d-4) & \textnormal{ if $d$ odd and $d<n-2$}\\
\frac{d(n-1)}{2}-1 & \textnormal{ otherwise}
\end{cases}
$$
Further, $R_{q,\frac{n-1}{2},n,d}=-\infty$ for $q\leq n-5$.
\end{lemma}

\begin{proof}
We first show the second assertion. Assume that $k=(n-1)/2$ and $q\leq n-5$. Then since $y_{n-q}=0$ we have
$$
(n-q-3)y_3 \geq |\underline{y}|-2y_1 \geq d(2k-q-2)= d(n-q-3),
$$
so $y_3\geq d$, contradicting (\ref{firsttermy}). Therefore $\mathcal{YU}(q, \frac{n-1}{2},n,d)$ is empty for $q\leq n-5$.

To verify the first assertion, note that by Corollary \ref{regularityFormula}, we have that $R_{n-3,\frac{n-1}{2},n,d}=\textnormal{reg}(S/I_{n-1}^d)$. By (\ref{subscriptpfaff}), it follows that $I_{n-1}^d=I_{\underline{x}}$ where $\underline{x}=(d^{(n-1)/2})=(d^m)$. By Theorem \ref{thmIlambdaodd} the result follows.
\end{proof}

To prove Proposition \ref{optimaleven} in cases (1) and (2), we first verify that $R_{q,k,n,d}$ is bounded below by the claimed value:

\begin{lemma}\label{lowerbound}
If $q$ is even, $0\leq q<2k\leq n-2$, and $d\geq n-3$, then 
$$
R_{q,k,n,d}\geq dk-1+q\left(k-\frac{q}{2}-1\right).
$$
\end{lemma}

\begin{proof}
Let $y_1=\cdots = y_{2k-q}=d-1$, $y_{2k-q+1}=y_{2k-q+2}=k-q/2-1$, and $y_i=0$ for $i>2k-q+2$, and let $u_1=\cdots =u_{2k-q}=l$, $u_i=0$ for $i>2k-q$. It is easy to show that $(\underline{y},\underline{u})\in \mathcal{YU}(q,k,n,d)$. We observe that
$$
g_{q,k,n,d}(\underline{y},\underline{u})=\frac{d(2k-q)-2}{2}+\frac{q(2k-q)}{2}+\frac{q(d-2)}{2}= dk-1+q\left( k-\frac{q}{2}-1\right).
$$
Since $R_{q,k,n,d}\geq g_{q,k,n,d}(\underline{y},\underline{u})$, the result follows.
\end{proof}

Finally, we complete the proof of Proposition \ref{optimaleven}:

\begin{proof}
[Proof of Proposition \ref{optimaleven}] By Lemma \ref{maximalpfaffeven} and Lemma \ref{maximalpfaffodd} we only need to check two cases: 
\begin{enumerate}
\item $n$ is even, $q$ is even, $0\leq q<2k\leq n-2$, and $d\geq n-2$
\item $n$ is odd, $q$ is even, $0\leq q<2k\leq n-2$, and $d\geq n-3$
\end{enumerate}

By Lemma \ref{lowerbound}, it suffices to show that
\begin{equation}\label{lizard}
R_{q,k,n,d}\leq dk-1+q\left(k-\frac{q}{2}-1\right),
\end{equation}
in these two cases. Among the elements $(\underline{y},\underline{u})\in \mathcal{YU}(q,k,n,d)$ for which $g_{q,k,n,d}(\underline{y},\underline{u})=R_{q,k,n,d}$, we consider one for which $\underline{y}$ is lexicographically maximal. The reader may verify that
\begin{equation}\label{sizey}
|\underline{y}|=d(2k-q)-2.
\end{equation}
To prove the proposition we proceed by induction on $n$. We divide our analysis into four cases:

\textbf{Case 1:} $u_{n-q}=y_{n-q}=0$. We can think of $\underline{y}$ as an element of $\mathcal{P}_e(n-q-1)$ and $\underline{u}$ as an element of $\mathcal{P}(n-q-1)$. It follows from (\ref{YU}) that $(\underline{y},\underline{u})\in \mathcal{YU}(q,k,n-1,d)$. Therefore, by induction we obtain:
$$
R_{q,k,n,d}=g_{q,k,n,d}(\underline{y},\underline{u})=g_{q,k,n-1,d}(\underline{y},\underline{u})\leq R_{q,k,n-1,d}.
$$
If $2k=n-2=(n-1)-1$, then since $\mathcal{YU}(q,k,n-1,d)$ is nonempty, we must have that $q=(n-1)-3=n-4$ by Lemma \ref{maximalpfaffodd}. In this case $n$ is even, and since $d\geq n-2=(n-1)-1$, we have
$$
R_{q,k,n-1,d} \overset{\textnormal{Lemma}\; \ref{maximalpfaffodd}}=\frac{d((n-1)-1)}{2}-1=dk-1+q\left(k-\frac{q}{2}-1\right).
$$
If $2k<n-2$, then since $d\geq n-3=(n-1)-2$ (when $n$ is even or odd), it follows from induction on $n$ that $R_{q,k,n-1,d}=dk-1+q(k-q/2-1)$. In both cases, we conclude that the proposition holds.

\textbf{Case 2:} $u_{n-q}=0$ and $y_{n-q}>0$. The proof of this case goes as in the proof of Case 2 in \cite[Proposition 4.1]{raicu2016regularity}, and is left to the interested reader.

\textbf{Case 3:} $u_{n-q}=p>0$ and $y_{n-q}\leq (p-n+d+1)+2(k-q/2-1)$. Since $0\leq p \leq q$ is even, there is a bijection
$$
\left\{ (\underline{z},\underline{w})\in \mathcal{YU}(q,k,n,d)\mid w_{n-q}\geq p\right\} \longleftrightarrow \mathcal{YU}\left(q-p,k-\frac{p}{2}, n-p, d\right),
$$
given by
$$
(\underline{z},\underline{w})\longleftrightarrow (\underline{z}, \underline{w}-(p^{n-q})).
$$
Let $\underline{v}=\underline{u}-(p^{n-q})$. Then $(\underline{y},\underline{v})\in \mathcal{YU}(q-p,k-p/2,n-p,d)$, and
$$
g_{q-p,k-p/2,n-p,d}(\underline{y},\underline{v}) = g_{q,k,n,d}(\underline{y},\underline{u})-\frac{p}{2}(n-p+y_{n-q}-1).
$$
Therefore,
\begin{equation}\label{case3}
g_{q,k,n,d}(\underline{y},\underline{u})\leq R_{q-p,k-\frac{p}{2},n-p,d}+\frac{p}{2}(n-p+y_{n-q}-1)\leq R_{q-p,k-\frac{p}{2},n-p,d}+\frac{pd}{2}+p(k-\frac{q}{2}-1).
\end{equation}
Since $d\geq n-3\geq (n-p)-2$ and $q-p<2k-p\leq (n-p)-2$, by induction we obtain:
$$
R_{q-p,k-\frac{p}{2},n-p,d}=d\left(k-\frac{p}{2}\right)-1+(q-p)\left(k-\frac{q}{2}-1\right).
$$
Thus, by (\ref{case3}) we have that $g_{q,k,n,d}(\underline{y},\underline{u})$ is bounded above by the desired value, verifying $(\ref{lizard})$ in this case.

\textbf{Case 4:} $u_{n-q}=p>0$, $y_{n-q}> (p-n+d+1)+2(k-q/2-1)$. Since $0\leq q<2k\leq n-2$ and $d\geq n-3$ it follows that
$$
(p-n+d+1)+2(k-q/2-1)\geq (2-n+(n-3)+1)=0.
$$
Thus, $y_{n-q}\geq 1$ and since $\underline{y}\in \mathcal{P}_e(n-q)$ it follows that $n$ is even. Notice that $u_{n-q}-y_{n-q}\leq (n-d-2)+2(q/2-k+1)$. By the conditions (\ref{YU}), we have that $u_i-u_{i+1}\leq y_{i}-y_{i+1}$ for $i=1,\cdots, n-q-1$. It follows that $u_i-y_i\leq u_{i+1}-y_{i+1}$ for $i=1,\cdots, n-q-1$, so $u_i-y_i\leq (n-d-2)+2\left(\frac{q}{2}-k+1\right)\textnormal{ for $i\leq n-q$}$. Since $u_1=q$ adding these inequalities for $i=3,\cdots , n-q$ gives
\begin{equation}\label{case4}
|\underline{u}|\leq (|\underline{y}|-2y_1)+2q+(n-q-2)\left((n-d-2)+2\left(\frac{q}{2}-k+1\right)\right).
\end{equation}
The proof proceeds as in the proof of Case 4 in \cite[Proposition 4.1]{raicu2016regularity}, and it uses that $d\geq n-2$.
\end{proof}

\subsection{Conclusion of the regularity computations}\label{proofReg}

We are now ready to prove Theorem A. Assume as before that $2\leq 2k\leq n$, and identify $S\cong \mathbb{C}[x_{i,j}]_{1\leq i<j\leq n}$. Recall that $I_{2k}$ is the ideal of $2k\times 2k$ Pfaffians of the generic skew-symmetric matrix $(x_{i,j})$. For ease of notation we again write $m=\lfloor n/2 \rfloor$.

Using the notation of (\ref{subscriptpfaff}), let $\mathcal{Z}_{k}^d=\mathcal{Z}\left(\mathcal{X}_{k}^d\right)$. By \cite[Lemma 5.3]{raicu2016regularity} we have the explicit description:
\begin{equation}\label{Zpfaffequation}
\mathcal{Z}_{k}^d=\left\{ (\underline{z},l)\in \mathcal{P}(m)\times \mathbb{Z} \mid \substack{ 0\leq l\leq k-1,\;\; z_1=\cdots=z_{l+1}\leq d-1\\ |\underline{z}|+(d-z_1)l+1\leq kd\leq |\underline{z}|+(d-z_1)(l+1)        }\right\}.
\end{equation}

\noindent The following lemma and corollary tie the values $R_{q,k,n,d}$ from the previous subsection to the regularity of the ideals $I_{2k}^d$, $I_{2k}^{(d)}$, and $(I_{2k}^d)^{\tn{sat}}$. This allows us to compute the regularity of these ideals via the optimization result Proposition \ref{optimaleven} from the previous subsection.

\begin{lemma}\label{regisR}
For each $l$ with $0\leq l\leq k-1$ we have using the notation of (\ref{Rfunction}) the equality
\begin{equation}\label{Requiv}
\max\left\{ \textnormal{reg}\left(J_{\underline{z},l}\right)\mid \underline{z}\in \mathcal{P}(m)\textnormal{  and  }(\underline{z},l)\in \mathcal{Z}_{k}^d\right\}=R_{2l,k,n,d}.
\end{equation}
\end{lemma}

\begin{proof}
The proof is similar to the proof of \cite[Lemma 5.4]{raicu2016regularity} and proceeds by showing the equivalence
$$
(\underline{z},l)\in \mathcal{Z}_k^d\textnormal{ and }\underline{t}\in \mathcal{T}_l(\underline{z})\iff (\underline{y},\underline{u})\in \mathcal{YU}(2l,k,n,d),
$$
where $\underline{y}=(z_{l+1},\cdots , z_m)^{(2)}$ and $u_i=2t_i$ for all $i$. Under this equivalence we have that 
$$
g_{2l,k,n,d}(\underline{y},\underline{u})=\sum_{i=l+1}^m z_i+l(z_1-1)+|\underline{t}|-f_l(\underline{z},\underline{t}).
$$
Via this equivalence, the result follows from Theorem \ref{regJ}.
\end{proof}

\begin{corollary}\label{regularityFormula}
For every $2\leq 2k \leq n$ and $d\geq 1$, we have
$$
\textnormal{reg}\left( S/I_{2k}^d\right)= \max_{\substack{0\leq q\leq 2k-2\\ q \textnormal{ even}  }} R_{q,k,n,d},\; \textnormal{reg}\left( S/(I_{2k}^d)^{\textnormal{sat}}\right)= \max_{\substack{2\leq q\leq 2k-2\\ q \textnormal{ even}  }} R_{q,k,n,d},\;\textnormal{and }\textnormal{reg}\left( S/I_{2k}^{(d)}\right)= R_{2k-2,k,n,d}.
$$
In particular, $\textnormal{reg}\left(I_{2k}^d\right)\geq \textnormal{reg}\left((I_{2k}^d)^{\textnormal{sat}}\right)\geq \textnormal{reg}\left( I_{2k}^{(d)}\right)$.
\end{corollary}

\begin{proof}
The proof is similar to the proof of \cite[Corollary 5.5]{raicu2016regularity}, and uses (\ref{extformula}) and Lemma \ref{regisR}.
\end{proof}

From Corollary \ref{regularityFormula} and Proposition \ref{optimaleven}, we immediately obtain (\ref{lineareqReg}) in the statement of Theorem A. The final two assertions of Theorem A follow from Corollary \ref{regularityFormula} and the following two lemmas:

\begin{lemma}\label{smallpowersformula}
If $2<2k\leq n-2$ and $1\leq d \leq n-4$ then
$$
R_{2k-2,k,n,d}\geq dk.
$$
\end{lemma}

\begin{proof}
We break the proof into 2 cases, when $d$ is even and when $d$ is odd.

\textbf{Case 1:} $d$ is even. We define $(\underline{y},\underline{u})\in \mathcal{P}_e(n-2k+2)\times \mathcal{P}(n-2k+2)$ via $y_1=y_2=d-1$, $y_i=0$ for $i>2$, and $u_1=u_2=2k-2$, $u_i=\max(2,2k-d)$ for $i>2$. Then
$$
|\underline{y}|=2d-2=d(2k-(2k-2))-2,\textnormal{ and } |\underline{y}|-2y_1=0=d(2k-(2k-2)-2).
$$
Also, $y_2-y_3=d-1\geq \min(2k-4, d-2)=u_2-u_3$ and $y_i-y_{i+1}=u_i-u_{i+1}=0$ for all $i\neq 2$. Thus, $(\underline{y},\underline{u})\in \mathcal{YU}(2k-2,k,n,d)$. We conclude that $R_{2k-2,k,n,d}\geq g_{2k-2,k,n,d}(\underline{y},\underline{u})$. We will show that $g_{2k-2,k,n,d}(\underline{y},\underline{u})\geq dk$.

Suppose first that $2k-d\geq 4$, in which case $u_i=2k-d$ for $i>2$. Then
$$
g_{2k-2,k,n,d}(\underline{y},\underline{u}) = dk+1+\frac{(n-2k)(2k-d)-4}{2}-\frac{(2k-d)}{2}.
$$
Thus, we need $(n-2k-1)(2k-d)-4\geq -2$. Since $n-2k-1\geq 0$ and $d\leq 2k-4$ we have that this holds. Now suppose that $2k-d\leq 2$, so that $u_i=2$ for $i>2$. Then
$$
g_{2k-2,k,n,d}(\underline{y},\underline{u}) = dk+1+\frac{2(n-2k)-4}{2}-(d-2k+3).
$$
Thus, we need $2(n-2k)-4-2(d-2k+3)\geq -2$. Since $d\leq n-4$ this is easy to check.

\textbf{Case 2:} $d$ is odd.  We define $(\underline{y},\underline{u})\in \mathcal{P}_e(n-2k+2)\times \mathcal{P}(n-2k+2)$ via $y_1=y_2=d-1$, $y_i=0$ for $i>2$, and $u_1=u_2=2k-2$, $u_i=\max(2,2k-d-1)$ for $i>2$. The proof then goes as in the previous case.
\end{proof}

\begin{lemma}\label{smallpowersformula1}
If $n$ is even, $2<2k\leq n-2$, and $d=n-3$, then 
$$
R_{2k-4,k,n,d}\geq dk.
$$
\end{lemma}

\begin{proof}
We define $(\underline{y},\underline{u})\in \mathcal{P}_e(n-2k+4)\times \mathcal{P}(n-2k+4)$ via $y_1=\cdots =y_4=d-1=n-4$, $y_5=y_6=1$, $y_i=0$ for $i>6$, and $u_1=\cdots =u_4=2k-4$, $u_i=0$ for $i>4$. Then $(\underline{y},\underline{u})\in \mathcal{YU}(2k-4,k,n,d)$ and 
$$
R_{2k-4,k,n,d}\geq g_{2k-4,k,n,d}(\underline{y},\underline{u})=\frac{1}{2}(4(n-3)-2+(2k-4)(d+2))\geq (n-3)k=dk.
$$
\end{proof}

To prove Theorem B, we need to characterize when $\tn{reg}(I_{2k}^d)=kd$. When $2<2k\leq n-2$, Theorem A implies that $I_{2k}^d$ has linear minimal free resolution if and only if $k=2$ and one of the following holds: (a) $n$ is even, $d\geq n-2$, (b) $n$ is odd, $d\geq n-3$. Therefore, to complete the proof of Theorem B, it suffices to study the cases $2k=2$, $2k=n-1$, and $2k=n$. Using Corollary \ref{regularityFormula}, Lemma \ref{maximalpfaffeven}, and Lemma \ref{maximalpfaffodd}, we obtain:

\begin{theorem}\label{outlierReg}
If $2k=2$ or $2k=n$ then $\textnormal{reg}(I_{2k}^d)=\textnormal{reg}((I_{2k}^d)^{\textnormal{sat}})=\textnormal{reg}(I_{2k}^{(d)})=dk$ for all $d\geq 1$. If $2k=n-1$ then
$$
\textnormal{reg}\left(I_{2k}^d\right)=\textnormal{reg}\left(\left(I_{2k}^d\right)^{\textnormal{sat}}\right)=\textnormal{reg}\left(I_{2k}^{(d)}\right)=
\begin{cases}
dk+\frac{1}{2}(n-d-4) & \textnormal{ if $d$ is odd and $d<n-2$}\\
dk & \textnormal{ otherwise.}
\end{cases}
$$
\end{theorem}

\noindent Note that by (\ref{subscriptpfaff}), (\ref{subscriptsymbolic}), and (\ref{satPfaff}), we have that $I_n^d=(I_n^d)^{\tn{sat}}=I_n^{(d)}$ and $I_{n-1}^d=(I_{n-1}^d)^{\tn{sat}}=I_{n-1}^{(d)}$ for all $d\geq 1$. The regularity computations of $I^d_{n-1}$ are originally due to Boffi-S\'{a}nchez \cite{boffi1992resolutions} and independently by Kustin-Ulrich \cite{kustin1992family}. The ideal $I_{n}$ is principal, generated by the $n\times n$ Pfaffian, so the regularity computations of powers of $I_n$ are clear. Theorem \ref{outlierReg} finishes the proof of Theorem B.

We now recall a fact which shows that the modules $\textnormal{Ext}^{\bullet}_S(S/I_{2k}^{(d)},S)$ grow with $d$. If we let $\mathcal{Z}_{k}^{(d)}=\mathcal{Z}(\mathcal{X}_{k}^{(d)})$ then by \cite[Lemma 5.8]{raicu2016regularity} we have
\begin{equation}\label{symbolicdirectequation}
\mathcal{Z}_{k}^{(d)}=\left\{ (\underline{z},k-1) \mid \underline{z}\in \mathcal{P}(m), z_1=\cdots=z_{k}, z_{k}+z_{k+1}+\cdots +z_m\leq d-1\right\}.
\end{equation}

\noindent Since $\mathcal{Z}_k^{(d)}\subset \mathcal{Z}_k^d$ and $\mathcal{Z}_{k}^{(d)}\subset \mathcal{Z}_{k}^{(d+1)}$ for all $d$, an immediate application of part (\ref{kernel}) of Theorem F yields Theorem G. As a corollary, we may characterize invariant ideals $I\subseteq S$ for which the natural morphisms $\textnormal{Ext}^j_S(S/I,S)\to H_I^j(S)$ are injective for all $j\geq 0$, addressing a question of Eisenbud-Musta\c{t}\u{a}-Stillman in our case \cite[Question 6.2]{eisenbud2000cohomology}. In Example 6 of the same paper, the authors show that the ideal $I$ is necessarily \defi{unmixed}, i.e. all of the associated primes of $I$ are minimal. By a proof similar to the proof of \cite[Corollary 5.10]{raicu2016regularity}, we obtain a converse in our case:

\begin{corollary}\label{unmixed}
Let $I\subseteq S$ be a $\tn{GL}$-invariant ideal that is unmixed. The natural maps 
$$
\textnormal{Ext}^j_S(S/I,S)\longrightarrow H_I^j(S),
$$
are injective for all $j\geq 0$.
\end{corollary}

\section{Kodaira Vanishing for Pfaffian Thickenings}\label{Section6}

In \cite{bhatt2016stabilization}, Bhatt-Blickle-Lyubeznik-Singh-Zhang prove a version of the Kodaira vanishing theorem for the thickenings of local complete intersections which are defined by a power of the ideal sheaf. While they also show that the statement is false for more general thickenings, Raicu has shown that an analogue of their vanishing result holds for arbitrary $\textnormal{GL}_n(\mathbb{C})\times \textnormal{GL}_m(\mathbb{C})$-equivariant thickenings of determinantal varieties of $n\times m$ generic matrices \cite[Theorem 6.1]{raicu2016regularity}. In this section we prove Theorem E, the similar result for GL-equivariant Pfaffian thickenings, via graded local duality, using Theorem F and Theorem \ref{extjxp}.

\begin{proof}
[Proof of Theorem E] Write $d=\binom{n}{2}$. The beginning of the proof is identical to the proof of \cite[Theorem 6.1]{raicu2016regularity}: using graded local duality \cite[Theorem 3.6.19]{bruns1998cohen} we have isomorphisms of finite dimensional vector spaces
$$
H^q(Y,\mathcal{O}_Y(-j))\cong \textnormal{Ext}^{d-1-q}_S(S/I,S)_{-d+j}, \textnormal{ for all $q$.}
$$
Thus, in order to prove the desired vanishing statement it is enough to check by Theorem F that
\begin{equation}\label{extvanishing}
\textnormal{Ext}^{d-1-q}_S(J_{\underline{z},l},S)_{p}=0 \textnormal{ for $p>-d$ and $q<2n-4$.}
\end{equation}
Suppose that $\lambda\in \mathbb{Z}^n_{\textnormal{dom}}$ is a dominant integral weight such that $S_{\lambda}W^{\ast}$ appears as a subrepresentation of $\textnormal{Ext}^{d-1-q}_S(J_{\underline{z},l},S)_{p}$. Then by Theorem \ref{extjxp} we have that $p=-|\lambda|/2$, and $|\lambda|< n(n-1)$. Also by Theorem \ref{extjxp} there exists $\underline{t}\in \mathcal{T}_l(\underline{z})$ such that $\lambda\in W(\underline{z},l,\underline{t})$. If $t_{n-2l}=0$ or $l=0$, then $\lambda_n=z^{(2)}_n+n-1$, so that $|\lambda|\geq n(n-1)+nz^{(2)}_n\geq n(n-1)$, a contradiction. Thus, we may assume that $l\geq 1$ and  $t_i\geq 1$ for all $i=1,\cdots ,n-2l$. We know that by Theorem \ref{extjxp},
$$
\binom{n}{2}-1-q=\binom{n}{2}-\binom{2l}{2}-2\sum_{i=1}^{n-2l}t_i.
$$
Notice that
$$
-\binom{2l}{2}-2\sum_{i=1}^{n-2l} t_i\leq -\binom{2l}{2}-2(n-2l)=-2n-2l^2+5l\leq -2n+3.
$$
Thus, $-1-q\leq -2n+3$, i.e. $q\geq 2n-4$, which proves (\ref{firstassertion}).

For the second assertion we note that $Y_{\textnormal{red}}$ is non-singular if and only if its defining ideal is $I_4$, in which case it is isomorphic to the Grassmannian $\mathbb{G}(2,n)$. Since $\textnormal{dim}(\mathbb{G}(2,n))=2n-4$, we see that the result holds in this case. If the defining ideal of $Y_{\textnormal{red}}$ is $I_{2k}$ with $2k\geq 6$, then $\textnormal{codim}(\textnormal{Sing}(Y_{\textnormal{red}}))= 2n-4k+5<2n-4$.
\end{proof}

\section{Example: thickenings of the Grassmannian $\mathbb{G}(2,6)\subseteq \mathbb{P}(\bw^2 \mathbb{C}^6)$}\label{Section7}

In order to illustrate the use of Theorem F and Theorem \ref{extjxp}, we consider two equivariant thickenings of the Pl\"{u}cker-embedded Grassmannian $\mathbb{G}(2,6)\subseteq \mathbb{P}(\bw^2 \mathbb{C}^6)$. Let $W$ be a six-dimensional complex vector space, and write $S=\Sym(\bw^2 W)$. The Grassmannian $\mathbb{G}(2,6)$ is defined by $I_4\subseteq S$, the prime ideal of $4\times 4$ Pfaffians, whereas the ideals $I_{(2,1,0)}$ and $I_4^2$ define scheme-theoretic thickenings of $\mathbb{G}(2,6)$. By (\ref{subscriptpfaff}) we have that $I_4^2=I_{\mathcal{Y}}$, where
\begin{equation}
\mathcal{Y}=\left\{ (2,2,0), (2,1,1)\right\}\subset \mathcal{P}(3).
\end{equation}
For ease of notation, we set $\underline{x}=(2,1,0)\in \mathcal{P}(3)$. Since $\underline{x}\leq (2,2,0)$ and $\underline{x}\leq (2,1,1)$, it follows from (\ref{idealDecomp}) that $I_4^2\subset I_{\underline{x}}$. We will complete the following:
\begin{enumerate}
\item compute $\tn{Ext}^j_S(S/I_{\underline{x}},S)$ and $\tn{Ext}^j_S(S/I_4^2,S)$ for all $j\geq 0$, using (\ref{extformula}) and Theorem \ref{extjxp},

\item use the computation of $\tn{Ext}^{\bullet}_S(S/I_4^2,S)$ from (1) to obtain the Castelnuovo-Mumford regularity of $I_4^2$,

\item determine the morphisms 
\begin{equation}
\tn{Ext}^j_S(S/I_{\underline{x}},S)\to \tn{Ext}^j_S(S/I_4^2,S),
\end{equation}
induced by the inclusion $I_4^2\subset I_{\underline{x}}$, for all $j\geq 0$.

\end{enumerate}
By Remark \ref{subquotsilambda} and (\ref{Zpfaffequation}) we have that
\begin{equation}
\mathcal{Z}(\underline{x})=\{ (\underline{0},1), ((1,1),0), ((1,1,1),0)\},\;\;\; \mathcal{Z}(\mathcal{Y})=\{ (\underline{0},1),((1,1),1), ((1,1,1),0)\}.
\end{equation}
In other words, $S/I_{\underline{x}}$ has a finite filtration by $S$-modules with quotients $J_{\underline{0},1}$, $J_{(1,1),0}$, and $J_{(1,1,1),0}$, each with multiplicity one (for more details, see Section \ref{mainsub}). Similarly, $S/I_4^2$ has finite filtration by $S$-modules with quotients $J_{\underline{0},1}$, $J_{(1,1),1}$ and $J_{(1,1,1),0}$, each with multiplicity one. Among these subquotients, the only non-vanishing Ext modules are:
\begin{equation}\label{extexamp1}
\begin{split}
\tn{Ext}^6_S(J_{\underline{0},1},S) & =\bigoplus_{\alpha\leq 3} S_{(3,3,3,3,\alpha,\alpha)}W^{\ast},\;\;\;\tn{Ext}^{15}_S(J_{(1,1),0},S)=S_{(6,6,6,6,5,5)}W^{\ast},\\ \tn{Ext}^{15}_S(J_{(1,1,1),0},S) & =S_{(6,6,6,6,6,6)}W^{\ast},\;\;\;\;\;\; \tn{Ext}^6_S(J_{(1,1),1},S)=\bigoplus_{\alpha\leq 3} S_{(4,4,3,3,\alpha,\alpha)}W^{\ast}.
\end{split}
\end{equation}

\begin{remark}\label{weymanbook}
By Remark \ref{pfaffisJ}, $J_{\underline{0},1}$ is isomorphic to $S/I_4$, the quotient of $S$ by the ideal of $4\times 4$ Pfaffians. Thus, the computation (\ref{extexamp1}) recovers the well-known fact that $S/I_4$ is Cohen-Macaulay \cite[Theorem 6.4.1(a)]{weyman2003cohomology}. Further, from (\ref{extexamp1}), one may also recover that the regularity of $S/I_4$ is $3$ \cite[Theorem 6.4.1(c)]{weyman2003cohomology}.
\end{remark}

We only verify the computation of $\tn{Ext}^{\bullet}_S(J_{\underline{0},1},S)$, leaving the remaining computations to the reader. Set $\underline{z}=\underline{0}$ and $l=1$. Using the notation of Theorem \ref{extjxp}, we have:
$$
T_l(\underline{z})=\{ \underline{t}=(1=t_1\geq t_2\geq t_3\geq t_4)\in \mathbb{Z}^4_{\geq 0}\mid t_1=t_2=t_3=t_4\}=\{(1,1,1,1)\}.
$$
As $|T_l(\underline{z})|=1$, it follows from (\ref{extjxp1}) that $\tn{Ext}^j_S(J_{\underline{0},1},S)$ is non-zero only if
$$
j=\binom{6}{2}-\binom{2}{2}-8=15-1-8=6.
$$
By (\ref{weightset}), we have that 
$$
W(\underline{0},1,(1,1,1,1))=\left\{ \lambda \in \mathbb{Z}^6_{\tn{dom}}\mid \substack{ \lambda_{i}=3\;\tn{ for $i=1,2,3,4$}\\ \lambda_{2i}=\lambda_{2i-1}\;\tn{for $i=1,2,3$}             }\right\}=\{ (3,3,3,3,\alpha,\alpha)\mid \alpha\leq 3\},
$$
recovering the computation of $\tn{Ext}^{\bullet}_S(J_{\underline{0},1},S)$ in (\ref{extexamp1}).

We conclude from (\ref{extformula}) that the only non-vanishing Ext modules of $S/I_{\underline{x}}$ and $S/I_4^2$ are described as follows:
\begin{equation}\label{extx}
\begin{split}
\tn{Ext}^6_S(S/I_{\underline{x}},S)\cong \tn{Ext}^6_S(J_{\underline{0},1},S),\;\;\; \tn{Ext}^{15}_S(S/I_{\underline{x}},S)\cong & \; \tn{Ext}^{15}_S(J_{(1,1),0},S)\oplus \tn{Ext}^{15}_S(J_{(1,1,1),0},S),\\
\tn{Ext}^6_S(S/I_4^2,S)\cong \tn{Ext}^6_S(J_{\underline{0},1},S)\oplus \tn{Ext}^6_S(J_{(1,1),1},S),\;\;\; \tn{Ext}^{15}_S & (S/I_4^2,S)\cong \tn{Ext}^{15}_S(J_{(1,1,1),0},S),
\end{split}
\end{equation}
where ``$\cong$" denotes an isomorphism of representations of $\tn{GL}(W)$. This completes part (1) of our analysis. 

We now complete part (2) of our analysis: calculating the regularity of $S/I_4^2$. By Theorem B, the ideal $I_4^2$ does not have a linear minimal free resolution. Since $I_4^2$ is generated in degree $4$, and $\tn{reg}(I_4^2)=\tn{reg}(S/I_4^2)+1$, it follows that $\tn{reg}(S/I_4^2)\geq 4$. Computing the Betti table of $S/I_4^2$ in Macaulay2 \cite{M2}, we see that its regularity is equal to $4$, a fact that we now verify with the computations (\ref{extexamp1}) and (\ref{extx}): by (\ref{compreg}) and (\ref{extformula}) we obtain
$$
\tn{reg}(S/I_4^2)=\max\left\{\tn{reg}(J_{\underline{0},1}), \tn{reg}(J_{(1,1),1}), \tn{reg}(J_{(1,1,1),0})\right\},
$$
where $\tn{reg}(J_{\underline{0},1})=\tn{reg}(S/I_4)=3$ (see Remark \ref{weymanbook}), and
$$
\tn{reg}(J_{(1,1),1})=\max_{\alpha\leq 3}\left\{ \frac{|(4,4,3,3,\alpha,\alpha)|}{2}-6\right\}=4,\;\;\; \tn{reg}(J_{(1,1,1),0})=\frac{36}{2}-15=3.
$$
Therefore, $\tn{reg}(S/I_4^2)=\max\{3,4\}=4$.

Finally, we determine the morphisms
$$
\tn{Ext}_S^6(S/I_{\underline{x}},S)\overset{\delta_6}\longrightarrow \tn{Ext}_S^6(S/I_4^2,S),\;\;\;\;\; \tn{Ext}_S^{15}(S/I_{\underline{x}},S)\overset{\delta_{15}}\longrightarrow \tn{Ext}_S^{15}(S/I_4^2,S),
$$
induced by the inclusion $I_4^2\subset I_{\underline{x}}$. By (\ref{extexamp1}) and (\ref{extx}), the $S$-modules $\tn{Ext}^6_S(S/I_{\underline{x}},S)$ and $\tn{Ext}^6_S(J_{(1,1),1},S)$ share no irreducible $\tn{GL}(W)$-representations. As the morphism $\delta_6$ is $\tn{GL}(W)$-equivariant, Schur's Lemma implies that the image of $\delta_6$ is an $S$-submodule of $\tn{Ext}^6_S(S/I_4^2,S)$ that is a subrepresentation of $\tn{Ext}^6_S(J_{\underline{0},1},S)$. Since $(\underline{0},1)\in \mathcal{Z}(\underline{x})\cap \mathcal{Z}(\mathcal{Y})$, it follows from (\ref{image}) that there is a short exact sequence of $S$-modules:
$$
0\longrightarrow \tn{Ext}^6_S(S/I_{\underline{x}},S)\overset{\delta_6}\longrightarrow \tn{Ext}^6_S(S/I_4^2,S)\longrightarrow \tn{Ext}^6_S(J_{(1,1),1},S)\longrightarrow 0.
$$
We now study the map $\delta_{15}$, viewed as a morphism of representations of $\tn{GL}(W)$. Again, by Schur's Lemma, we know that
$\delta_{15}(\tn{Ext}^{15}_S(J_{(1,1),0},S))=0$. A priori, $\delta_{15}(\tn{Ext}^{15}_S(J_{(1,1,1),0},S))$ could be any subrepresentation of $\tn{Ext}_S^{15}(S/I_4^2,S)$. However, since $((1,1,1),0)\in \mathcal{Z}(\underline{x})\cap \mathcal{Z}(\mathcal{Y})$, it follows from (\ref{image}) that $\delta_{15}$ is surjective. In other words, there is a short exact sequence of $S$-modules:
$$
0\longrightarrow \tn{Ext}^{15}_S(J_{(1,1),0},S)\longrightarrow \tn{Ext}^{15}_S(S/I_{\underline{x}},S)\overset{\delta_{15}}\longrightarrow \tn{Ext}^{15}_S(S/I_4^2,S)\longrightarrow 0.
$$
This completes our analysis of these examples.

\section*{Acknowledgments}

The author is very grateful to Claudiu Raicu for his guidance while this work was done. Experiments were conducted using the computer algebra software Macaulay2 \cite{M2}. The author was supported by the National Science Foundation Graduate Research Fellowship under Grant No. DGE-1313583.

\bibliographystyle{alpha}
\bibliography{mybib}

\newcommand{\etalchar}[1]{$^{#1}$}
\begin{thebibliography}{BBL{\etalchar{+}}16}

\bibitem[ADF80]{abeasis1980young}
S.~Abeasis and A.~Del~Fra.
\newblock Young diagrams and ideals of {P}faffians.
\newblock {\em Adv. in Math.}, 35(2):158--178, 1980.

\bibitem[BBL{\etalchar{+}}16]{bhatt2016stabilization}
Bhargav Bhatt, Manuel Blickle, Gennady Lyubeznik, Anurag~K Singh, and Wenliang
  Zhang.
\newblock Stabilization of the cohomology of thickenings.
\newblock {\em arXiv preprint arXiv:1605.09492}, 2016.

\bibitem[BH93]{bruns1998cohen}
Winfried Bruns and J\"{u}rgen Herzog.
\newblock {\em Cohen-{M}acaulay rings}, volume~39 of {\em Cambridge Studies in
  Advanced Mathematics}.
\newblock Cambridge University Press, Cambridge, 1993.

\bibitem[BS92]{boffi1992resolutions}
Giandomenico Boffi and Rafael S\'{a}nchez.
\newblock On the resolutions of the powers of the {P}faffian ideal.
\newblock {\em J. Algebra}, 152(2):463--491, 1992.

\bibitem[CHT99]{cutkosky1999asymptotic}
S.~Dale Cutkosky, J\"{u}rgen Herzog, and Ng\^{o}~Vi\^{e}t Trung.
\newblock Asymptotic behaviour of the {C}astelnuovo-{M}umford regularity.
\newblock {\em Compositio Math.}, 118(3):243--261, 1999.

\bibitem[EMS00]{eisenbud2000cohomology}
David Eisenbud, Mircea Musta\c{t}\v{a}, and Mike Stillman.
\newblock Cohomology on toric varieties and local cohomology with monomial
  supports.
\newblock {\em J. Symbolic Comput.}, 29(4-5):583--600, 2000.
\newblock Symbolic computation in algebra, analysis, and geometry (Berkeley,
  CA, 1998).

\bibitem[GS]{M2}
Daniel~R. Grayson and Michael~E. Stillman.
\newblock Macaulay2, a software system for research in algebraic geometry.
\newblock Available at \url{http://www.math.uiuc.edu/Macaulay2/}.

\bibitem[ILL{\etalchar{+}}07]{iyengar2007twenty}
Srikanth~B. Iyengar, Graham~J. Leuschke, Anton Leykin, Claudia Miller, Ezra
  Miller, Anurag~K. Singh, and Uli Walther.
\newblock {\em Twenty-four hours of local cohomology}, volume~87 of {\em
  Graduate Studies in Mathematics}.
\newblock American Mathematical Society, Providence, RI, 2007.

\bibitem[JPW81]{jozefiak1981resolutions}
T.~J\'{o}zefiak, P.~Pragacz, and J.~Weyman.
\newblock Resolutions of determinantal varieties and tensor complexes
  associated with symmetric and antisymmetric matrices.
\newblock In {\em Young tableaux and {S}chur functors in algebra and geometry
  ({T}oru\'{n}, 1980)}, volume~87 of {\em Ast\'{e}risque}, pages 109--189. Soc.
  Math. France, Paris, 1981.

\bibitem[Kod00]{kodiyalam2000asymptotic}
Vijay Kodiyalam.
\newblock Asymptotic behaviour of {C}astelnuovo-{M}umford regularity.
\newblock {\em Proc. Amer. Math. Soc.}, 128(2):407--411, 2000.

\bibitem[KU92]{kustin1992family}
Andrew~R. Kustin and Bernd Ulrich.
\newblock A family of complexes associated to an almost alternating map, with
  applications to residual intersections.
\newblock {\em Mem. Amer. Math. Soc.}, 95(461):iv+94, 1992.

\bibitem[Lyu93]{MR1223223}
Gennady Lyubeznik.
\newblock Finiteness properties of local cohomology modules (an application of
  {$D$}-modules to commutative algebra).
\newblock {\em Invent. Math.}, 113(1):41--55, 1993.

\bibitem[Per18]{perlman2018lyubeznik}
Michael Perlman.
\newblock Lyubeznik numbers for {P}faffian rings.
\newblock {\em arXiv preprint arXiv:1809.06745}, 2018.

\bibitem[Rai18]{raicu2016regularity}
Claudiu Raicu.
\newblock Regularity and cohomology of determinantal thickenings.
\newblock {\em Proc. Lond. Math. Soc. (3)}, 116(2):248--280, 2018.

\bibitem[RW16]{raicu2016local}
Claudiu Raicu and Jerzy Weyman.
\newblock Local cohomology with support in ideals of symmetric minors and
  {P}faffians.
\newblock {\em J. Lond. Math. Soc. (2)}, 94(3):709--725, 2016.

\bibitem[RWW14]{raicu2014submax}
Claudiu Raicu, Jerzy Weyman, and Emily~E. Witt.
\newblock Local cohomology with support in ideals of maximal minors and
  sub-maximal {P}faffians.
\newblock {\em Adv. Math.}, 250:596--610, 2014.

\bibitem[Wey03]{weyman2003cohomology}
Jerzy Weyman.
\newblock {\em Cohomology of vector bundles and syzygies}, volume 149 of {\em
  Cambridge Tracts in Mathematics}.
\newblock Cambridge University Press, Cambridge, 2003.

\end{thebibliography}

\Addresses

\end{document}